\documentclass[11pt,tbtags,reqno]{article}
\usepackage[margin=0.94in]{geometry}
\usepackage{eulervm} 
\usepackage{amssymb}
\usepackage{amsmath}
\usepackage{amsthm}
\usepackage[swedish, english]{babel}
\usepackage{psfrag}   
\usepackage{graphicx}
\usepackage[retainorgcmds]{IEEEtrantools}
\usepackage{bbm}
\usepackage[square, comma, numbers, sort&compress]{natbib}
\usepackage{enumerate}
\usepackage{soul}
\usepackage{hyperref}
\usepackage[colorinlistoftodos, textsize=tiny, backgroundcolor=yellow
, linecolor=green]{todonotes}
\usepackage{sidecap}
\usepackage{float}
\usepackage{leftidx}

\usepackage{bm}

\newcommand{\E}{\mathbb E}
\newcommand{\R}{\mathbb R}
\newcommand{\N}{\mathbb N}

\newcommand{\F}{\mathcal F}

\newcommand{\G}{\mathcal G}

\newcommand{\eps}{\varepsilon}

\newcommand{\Ll}{\mathbb{L}}
\newcommand{\ud}{\mathrm{d}}

\def\P{{\mathbb P}}
\def\Q{{\mathbb Q}}

\numberwithin{equation}{section}
\theoremstyle{plain}

\newtheorem{theorem}{Theorem}[section]
\newtheorem{lemma}[theorem]{Lemma}
\newtheorem{corollary}[theorem]{Corollary}
\newtheorem{proposition}[theorem]{Proposition}

\newtheorem{remark}[theorem]{Remark}

\theoremstyle{definition}
 \theoremstyle{example}
  \theoremstyle{open question}

\begin{document}

\title{\textsc{
Necessary and Sufficient Conditions  for the Lacunary/Hereditary   Laws of Large Numbers} \thanks{~ The authors are grateful to Tomoyuki Ichiba and Nicola Doninelli for pointing out crucial results in this field.}
}

\author{  
\textsc{Istv\'an Berkes} \thanks{~
Institute of   Statistics,  Graz University of Technology, Kopernikusgasse 24, 8010 Graz, Austria  ({\it berkes@tugraz.at}) and Alfred R\'enyi Institute of Mathematics, Re\'altanoda utca 13-15, 1053 Budapest, Hungary ({\it berkes.istvan@renyi.hu}).
}  
 \and
\textsc{Ioannis Karatzas} \thanks{~
Departments of Mathematics and Statistics,  Columbia University, New York, NY 10027 ({\it ik1@columbia.edu}). Support from  the NSF     under Grant  DMS-25-06199, and from a Lenfest Award at  Columbia, is gratefully acknowledged.  
}  
 \and
\textsc{Walter Schachermayer}                \thanks{ \, 
Faculty of Mathematics, University of Vienna, Oskar-Morgenstern-Platz 1, 1090 Vienna,  Austria ({\it walter.schachermayer@univie.ac.at}). Support from the  Austrian Science Fund (FWF) under
 grants P-35519 and P-35197 is gratefully acknowledged.    
          }
                                      }

\maketitle

\begin{abstract}

\medskip
\noindent
The   celebrated 
theorem of \textsc{Koml\'os} \cite{K} asserts that $\Ll^1-$boundedness is  sufficient  for a given sequence of functions to contain a subsequence along which (in a ``lacunary" manner),  and along whose     every further subsequence (``hereditarily"),    a strong law of large numbers holds. We identify  here slightly weaker, \textsc{Egorov}-type conditions, as  not only sufficient in this context, but   necessary as well. Necessary and sufficient  conditions  are developed  also     for the lacunary/hereditary version  of the weak law of large numbers  for general sequences, as well as      for the weak law of large numbers   in the context of exchangeable sequences, both long-open questions. 
      \end{abstract}

\noindent
{\small {\sl AMS  2020 Subject Classification:}  Primary 60A10, 60F15; Secondary 60G57, 60G09, 60G42.}

\noindent
 {\small {\sl Keywords:}  Subsequences, hereditary convergence,  stable convergence, exchangeability, Koml\'os   theorem,   laws  of large numbers}


\section{Introduction}
\label{sec1}


The lacunary/hereditary strong law of large numbers established in the seminal paper of  \textsc{Koml\'os} \cite{K}, singles out   the boundedness-in-$\Ll^1$  condition 
\begin{equation} 
\label{1}
 \sup_{n \in \N} \,\E^\P \big( \big| f_n \big| \big) < \infty\,,
 \end{equation}
  pertaining to  a given sequence of measurable functions $\,f_1, f_2, \cdots\,$ on a probability space $(\Omega, \F, \P),$ as   sufficient for the existence of a measurable   $\,f_* : \Omega \to \R\,$ and of a subsequence $f_{k_1}, f_{k_2},  \cdots\,$    satisfying  the  convergence in \textsc{Ces\`aro} mean
\begin{equation} 
\label{2}
 \lim_{N \to \infty} \frac{1}{\,N\,} \sum_{n=1}^N f_{k_n}   = f_{*}     
\end{equation}
 $\P-$a.e.; as well as hereditarily, that is, along all further subsequences. The limiting function $f_*$ plays here the r\^ole of ``randomized mean", whose $\,\P-$integrability the condition \eqref{1} also guarantees.

 \bigskip
 \newpage
This is a striking,  and  very useful, result;   its proof  was  a major early success of martingale methods in classical probability limit theory. It extends the sufficiency statements, not only of the  \textsc{Kolmogorov}\,(\cite{K2},\,\cite{Kol}) strong law of large numbers, but also of the \textsc{Dunford\,-Pettis} (\cite{DP},\,\cite{DunSchs}) theorem in functional analysis. It has found numerous applications in that field  (e.g.,\,\cite{DS1},\,\cite{KK},\,\cite{LZ},\,\cite{Z}), as well as in stochastic analysis (e.g.,\,\cite{DS2},\,\cite{BSV1},\,\cite{BSV2},\,\cite{J}). It inspired the development of sufficient conditions for the lacunary/hereditary versions of the law of the iterated logarithm and of the central limit theorem, leading to   a heuristic ``principle of subsequences" (\cite{Ch1},\,\cite{G},\,\cite{Ch5}).

  Now it is well-known that, when the $f_1, f_2, \cdots$ are  independent and identically distributed (I.I.D.), the integrability condition  $\E \big( \big| f_1  \big| \big) < \infty$ is not only sufficient for the validity of the classical strong law of large numbers, but also necessary (e.g.\,\cite{Loe}, p.\,251). Likewise,   uniform integrability in the \textsc{Dunford\,-Pettis} theorem is not just    sufficient   for relative compactness in the space $\Ll^1$ equipped with the $\sigma (\Ll^1, \Ll^\infty)$ topology, but      necessary as well   (Theorem T23, page 20 in \cite{Mey}).  

  \smallskip
We identify here a condition,  slightly weaker than 
\eqref{1}, as  not only sufficient    for  the validity of the lacunary/hereditary  Strong Law of Large Numbers \eqref{2} with some measurable   $f_* : \Omega \to \R\,,$  but also necessary. Developed in Theorem   \ref{Thm2.1}, this  condition   posits the existence of sets  $\, A_1 \subseteq A_2 \subseteq \cdots \subseteq A_j \subseteq \cdots \,$ in $\,\F\,$ with $\lim_{j \to \infty} \P (A_j) =1,$  and  of a   subsequence   $\,f_{k_1},  f_{k_2}, \cdots\,, $  such that 
\begin{equation} 
\label{A1}
\Big( f_{k_n}  \, \mathbf{1}_{A_{j}} \Big)_{n \in \N}~~~\text{is bounded in} ~ \,\Ll^1 (\P)\,, ~\text{for each} ~~ j \in \N\,.
\end{equation} 
This   is equivalent to  the existence of a probability measure $\,\Q \sim \P\,,$ under which   a   subsequence $\,f_{k_1}, f_{k_2}, \cdots\,$ is bounded in $\, \Ll^1 (\Q)\,;$ whereas, the slightly stronger condition \eqref{A16} of Proposition  \ref{Prop2.4} is   necessary and sufficient for  the \textsc{Ces\`aro} limit  $f_*$ 
to be  $\,\P-$integrable. 

Similarly,    Theorem   \ref{Thm3.2}  develops  necessary and sufficient conditions       for the lacunary/hereditary version of the  \textsc{Kolmogorov-Feller} (\cite{K1},\,\cite{F})  {\it Weak Law of Large Numbers} in the context of general sequences, a pivotal result of the paper.   Its proof turns out to be considerably more involved  than the one for the Strong Law of Large Numbers. 

The arguments  rest   on three pillars: $(i)$   the notion and properties of stable convergence (cf.\,\cite{Ren}); \\ $(ii)$   \textsc{Aldous}'s profound subsequence theory in\,\cite{A}, using    exchangeability;  
and $(iii)$    {\it  conditions    necessary and sufficient for   the Weak Law of Large Numbers in the context of exchangeable sequences.} 

These  latter conditions settle another long-open question.   It is quite striking  that, the necessary and sufficient conditions for the exchangeable case,  are {\it not} the corresponding classical conditions for the I.I.D.\,\,case  conditioned on the tail $\sigma$-algebra of the sequence.

\section{Sufficiency and Necessity in the  Lacunary/Hereditary SLLN}
\label{sec2}

  Here is a reformulation of     \textsc{Koml\'os}'s  result \cite{K} in terms of conditions both sufficient and necessary. It is proved in section \ref{sec4}; a   result of   independent interest, crucial for the proof,  is   in section \ref{sec5}.

  \begin{theorem} {\bf Sufficiency/Necessity in the  Lacunary/Hereditary SLLN for General Sequences:}
  \label{Thm2.1}
  On a  probability space $\,(\Omega, \F, \P),$  consider  real-valued, measurable functions $f_1, f_2, \cdots$. 
  
  \smallskip
  \noindent
  {\bf (i)} Suppose that for some subsequence $\,f_{k_1}, f_{k_2}, \cdots\,$    and   sets $\, A_1 \subseteq A_2 \subseteq \cdots \subseteq A_j \subseteq \cdots \,$ in $\,\F$ with $\,\lim_{j \to \infty} \P (A_j ) =1\,,$ the condition \eqref{A1} holds.  

There exist then a measurable function $f_* : \Omega \to \R$ and a $($further, relabelled$)$ subsequence along which, and along whose every subsequence, the \textsc{Ces\`aro}-mean convergence in \eqref{2} holds   $\, \P-$a.e.

   \smallskip
 \noindent
  {\bf (ii)} Conversely, suppose that along some subsequence $f_{k_1}, f_{k_2}, \cdots\,$   and  each of its subsequences, the \textsc{Ces\`aro}-mean convergence \eqref{2} holds $\, \P-$a.e.\,for some measurable $f_* : \Omega \to \R\,.$   There exist then    sets $\, A_1 \subseteq A_2 \subseteq \cdots \subseteq A_j \subseteq \cdots \,$ in $\,\F$ with $\,\lim_{j \to \infty} \P (A_j ) =1\,,$  such that \eqref{A1} is satisfied.
  
  \medskip
\noindent
  {\bf (iii)} The validity of \eqref{A1} for some subsequence $f_{k_1}, f_{k_2}, \cdots\,$   and for  sets $\, A_1 \subseteq A_2 \subseteq \cdots \subseteq A_j \subseteq \cdots \,$ in $\,\F$ with $\,\lim_{j \to \infty} \P (A_j ) =1\,,$ is equivalent to the existence of a probability measure $\,\Q \sim \P $ on $(\Omega, \F)\,$ with the property  that some      subsequence $\,f_{k_1}, f_{k_2}, \cdots\,$   is bounded in $\, \Ll^1 (\Q)\,.$

    \end{theorem}

     \begin{remark} 
  {\rm
It makes good sense, that probability measures equivalent to $\,\P\,$ should enter the picture here,  as they do in part (iii) above: the property in question, almost-everywhere convergence to a real-valued limit,  is invariant under equivalent changes of probability measure. 

In the   special  context  of a necessary condition for the \textsc{Koml\'os} theorem with nonnegative functions, this feature was pointed out and established by  \textsc{von\,Weizs\"acker}  \cite{vW}.
  }
    \end{remark}

\subsection{Stable Convergence}
\label{sec2.8}
    
We shall endow the space $\,\Ll^0 \equiv \Ll^0 (\P)\,$ of real-valued, measurable functions on a   probability space  $\,(\Omega, \F, \P)$,   with the topology induced by the metric
  $\,
  \Ll^0 \times \Ll^0 \ni \big( f, g \big)  \, \longmapsto \, \E^\P \big( 1 \wedge | f-g| \big) \in [0,1]\,.
\,$
Convergence under this metric is equivalent to convergence in $\P-$measure (``in probability").   The space $\,\Ll^0\,$ and its topology depend only on the equivalence class,  modulo null sets,  to which   $\P$ belongs: they are the same for any   probability measure  $\,\Q \sim \P\, $ on $(\Omega, \F)\,.$ 

We call a set $\, \mathbb{G} \subset \Ll^0\,$ {\it bounded in probability}, if $\, \lim_{M \to \infty} \sup_{g \in \mathbb{G}} \P \,\big( |g| > M \big) =0\,$ holds.

      Let us consider  now  real-valued, measurable    $\,f_1, f_2, \cdots\,$  on  the probability space $\,(\Omega, \F, \P),$ along with   their tail $\sigma-$algebra 
     \begin{equation} 
\label{A6a}
 {\cal T} \, : = \bigcap_{n \in \N}  {\cal T}_n\,, \qquad {\cal T}_n := {\bm \sigma} \big(  f_{n},  f_{n+1} , \cdots \big)\,.
\end{equation}
Here is a basic result, which  goes  back to \textsc{R\'enyi} \cite{Ren}: {\it Every  bounded in probability sequence $\,f_1, f_2, \cdots$  contains                                                a    subsequence  $f_{k_1}, f_{k_2}, \cdots$ which is ``determining"}, i.e.,   satisfies the   {\it  stable convergence}  (or  extended  \textsc{Helly-Bray} property, proved in  \cite{BC}; see also\,\cite{AE},\,and\,Theorem 2.2 in\,\cite{BerRos})   
  \begin{equation} 
\label{B0}
\lim_{n \to \infty} \P \big( f_{k_n}  \le x, B \big)   = \int_B {\bm H} (x, \omega) \, \P (\mathrm{d}\, \omega) \,, \qquad \forall ~ ~B \in \F 
\end{equation} 
 at each point $x  $ in a countable,   dense  set ${\bm D} \subset \R\, . $ Here $\, \R \times \Omega \ni (x, \omega)\longmapsto {\bm H}  (x, \omega) \in [0,1]\, $ is a  limit- \\ ing random probability distribution function;  for $\,\P-$a.e.\,$\,\omega \in \Omega\,,$ it induces on ${\cal B} (\R)$ the probability measure   ${\bm \mu}_\omega  \equiv {\bm \mu} (\omega)\,,$  where $\omega \mapsto   {\bm \mu}_\omega$ is  measurable with respect to the ``tail $\sigma-$algebra" of \eqref{A6a}.

\subsection{Integrability of the Limit in Theorem \ref{Thm2.1} }
\label{sec2.4}

The  following result addresses the $\P-$integrability of   the \textsc{Ces\`aro} limit   (``randomized mean") $f_*\,$ in Theorem \ref{Thm2.1}.  Its  first  claim can be argued  as in  \S\,6.1.4 of \cite{BKS}; the second as in section \ref{sec4} here.  

 \begin{proposition} 
 \label{Prop2.4} On a   probability space $\,(\Omega, \F, \P)$  consider   real-valued, measurable functions $\,f_1, f_2, \cdots$. \\ 
{\bf (i)} Suppose that  the sequence $\,f_1, f_2, \cdots$ is bounded in $\,\Ll^0$ with ``determining" subsequence  $f_{k_1}, f_{k_2}, \cdots\,,$ whose limiting random distribution $\,\omega \mapsto   {\bm \mu}_\omega\,,$     as in and below  \eqref{B0}, satisfies 
  \begin{equation} 
\label{A15}
 \E^\P \int_{\R} \big| x \big|  \,  {\bm \mu}  (\ud x, \cdot)  \,<\,\infty\,.
  \end{equation}
  
Then  the \textsc{Ces\`aro}  convergence \eqref{2} holds $\,\P-$a.e.\,\,for the $\,\P-$integrable function   $\, f_* = \int_{\R}   x   \,     {\bm \mu}  ( \ud x, \cdot): \Omega \to \R\,;$      and there exist a $($further, relabelled$\,)$ subsequence $f_{k_1}, f_{k_2}, \cdots\,,$ as well  as a sequence of sets $\, A_1 \subseteq A_2 \subseteq \cdots \subseteq A_n \subseteq \cdots \,$ in $\,\F\,$ with $\,\lim_{n \to \infty} \P (A_n ) =1\,,$ such that the sequence
\begin{equation} 
\label{A16}
\Big( f_{k_n}  \, \mathbf{1}_{A_{n}} \Big)_{n \in \N}~~~\text{is bounded in} ~ \,\Ll^1 (\P) \,.
\end{equation} 
 {\bf (ii)} Conversely, suppose there exist a subsequence $f_{k_1}, f_{k_2}, \cdots\,,$ and a sequence of sets $\, A_1 \subseteq A_2 \subseteq \cdots \subseteq A_n \subseteq \cdots \,$ in $\,\F\,,$ satisfying $\,\lim_{n \to \infty} \P (A_n ) =1\,$ as well as the property \eqref{A16}. 

Then the \textsc{Ces\`aro}-mean     convergence \eqref{2} holds $\,\P\,-$a.e.\,along a $($relabelled,\,``determining"$\,)$ subsequence,   with a  limiting function $\,f_*$ which is $\,\P\,-$integrable$\,:  \, \E^\P \big( \big| f_* \big| \big) < \infty\,.$
 \end{proposition}

 The condition \eqref{A16} is slightly stronger than \eqref{A1}, though  still weaker than the $  \,\Ll^1 (\P)-$boundedness   $\, \sup_{n \in \N} \,\E^\P \big( \big| f_{k_n} \big|   \big) < \infty\,$ associated with the \textsc{Koml\'os} theorem as in \eqref{1}; it is the exact analogue, in the present context, of the \textsc{Egorov}-type condition in Theorem 2.2 of \cite{BKS}.

\section{Sufficiency and Necessity in the Lacunary/Hereditary WLLN}
\label{sec3}

 We move now from the Strong to the Weak Law of Large Numbers (WLLN), and formulate in Theorem \ref{Thm3.2} right below    conditions both sufficient and necessary for the validity of its Lacunary/Hereditary version in the context of general sequences. The proof of this result  is presented in section \ref{sec10}; it    relies very crucially on   the  novel  sufficient/necessary conditions in   Theorem \ref{Thm6.1}      pertaining to   the Weak Law of Large Numbers for Exchangeable Sequences, a heretofore open question.

\subsection{General Sequences}
\label{sec3.3}

  \begin{theorem} 
  \label{Thm3.2}  {\bf Sufficiency/Necessity in the  Lacunary/Hereditary   WLLN  for    General Sequences:}
   Consider a sequence of measurable functions $\,f_1, f_2, \cdots$ on a   probability space $\,(\Omega, \F, \P)$, and recall the tail $\sigma-$algebra of \eqref{A6a}. Then the following conditions are equivalent: 
   
     \smallskip
     \noindent 
     {\bf (i)} {\rm Lacunary-Hereditary WLLN:} There exist real-valued,  $\,{\cal T}-$measurable ``correctors" $D_1, D_2, \cdots\,$ with $\,\P \big( \big| D_N \big| \le N,\, \forall \, N \in \N \big)=1$ and a subsequence $\,f_{k_1}, f_{k_2}, \cdots\,  $       along which,   and   along whose every  subsequence,  
     \begin{equation} 
\label{D2a}
 \lim_{N \to \infty} \bigg(  \frac{1}{\,N\,} \sum_{n=1}^N f_{k_n}   - D_N \bigg) =\,0 \quad \text{holds in} ~\P\text{\,--\,probability.}    
\end{equation}
 {\bf (i)$^\prime$} {\rm Lacunary-Hereditary Conditional WLLN:} There exist real-valued $\,{\cal T}-$measurable ``correctors" $\,D_1, D_2, \cdots$ with $\,\P \big( \big| D_N \big| \le N,\, \forall \, N \in \N \big)=1\,$ and a subsequence $\,f_{k_1}, f_{k_2}, \cdots\, $     along which, and along whose every subsequence,   we have  for every $\, \eps >0 \,$ the convergence           \begin{equation} 
\label{D5}
 \lim_{N \to \infty} \P  \bigg(  \bigg|  \frac{1}{\,N\,} \sum_{n=1}^N f_{k_n}   - D_N \bigg| > \eps  \, \bigg| \,{\cal T} \bigg) =\,0 \quad \text{in} ~~\P\text{\,--\,probability.}    
\end{equation}
     {\bf (ii)} The sequence $\, f_1, f_2, \cdots\,$ contains a {\rm determining} subsequence $\, f_{k_1}, f_{k_2}, \cdots\,$   whose limit random probability distribution  ${\bm \mu} \,,$      as in and below \eqref{B0},   satisfies  in $\P-$probability 
       \begin{equation} 
\label{B00}
\lim_{N \to \infty}  \Big( N \, {\bm \mu}  \big(\R \setminus [-N,N]  \big) \Big) = \,0\,,\qquad \lim_{N \to \infty}  \frac{1}{\,N\,} \int_{[-N,N]} x^2  \,  {\bm \mu}  (\ud   x, \cdot \,) \,=\,0\,.
\end{equation} 
Each  corrector    $\,D_N \,$  in {\bf (i)},\,{\bf (i)}$^\prime$  can be taken     as the weak$\,-\,\Ll^2\,$ limit of      $ \, \big( \, \E \big(  f_{k_{n}} \, \mathbf{1}_{\{ | f_{k_{n}} | \le N \}} \,\big|\, {\cal T} \,\big) \big)_{n \in \N}\,;$  and $\,  \lim_{N \to \infty} \big( D_N  /  \sqrt{N\,} \, \big) =0\,$ holds in $\P-$probability.
  \end{theorem}

     \begin{remark} {\bf  A Sufficient Condition:} 
  {\rm It was  shown in \cite{KS} that   the \textsc{Koml\'os}-type condition
 $\,
 \lim_{N \to \infty} \big( N \cdot \,\sup_{n \in \N} \,\P \big( \big| f_n \big| > N \big) \big) =0\,,$ 
is {\it sufficient}  for the validity of the Lacunary-Hereditary    Weak Law of Large Numbers \eqref{D2a}. As we shall see in Proposition   \ref{Prop6.3}\,(${\mathfrak A}$), this sufficient condition is {\it not} necessary for this WLLN to hold. 
  } 
       \end{remark}

\section{The WLLN for Exchangeable Sequences}
\label{sec6}

 The proof of Theorem  \ref{Thm3.2}  relies crucially on   conditions not just sufficient, but also {\it necessary},  for the validity of the Weak Law of Large Numbers (WLLN) in the context of exchangeable    sequences\,---\,a  question that has been open for a long time (cf.\,\cite{St1} and its citations).  
 
 We develop    such conditions now. This is a considerable task. It    encompasses   Theorem \ref{Thm6.2} for the classical WLLN in the I.I.D.\,\,case;   and  Theorem \ref{Thm6.1}   for the  exchangeable case,  a central result  of the present paper.

\subsection{The Classical Case}
\label{sec6.01}

On a   probability space $\,(\Omega, \F, \P)$, let us   consider     real-valued, I.I.D.       $\, h_1, h_2, \cdots\,,$ and   adopt the notation    
\begin{align} 
\delta_t (\eps)  \, &:= \,\E   \big( h_1 \cdot  \mathbf{1}_{ \{ |h_1| \le \eps t \} }   \big)\,, \quad \delta_t  := \delta_t (1)
\label{6.6} \\
 \pi_N (\eps) \,&:= \,\P \bigg( \, \bigg| \frac{1}{\,N\,} \sum_{n=1}^N h_n - \delta_N \bigg| > \eps  \bigg), \quad  \pi_N :=  \pi_N (1)
 \label{6.9'} \\
\tau_t (\eps) \,&:= \,t \cdot \P \big( \,  |h_1| > \eps\, t     \,    \big)\,, \quad \tau_t := \tau_t (1) 
\label{6.7} \\
 \sigma_t (\eps) & \,:= \,\frac{1}{\,t\,} \cdot \E \big( \,h_1^2 \cdot   \mathbf{1}_{ \{ |h_1| \le \eps t \} }    \big)\,, \quad \sigma_t :=  \sigma_t (1) 
 \label{6.8} \\
 v_t (\eps) \,&:= \, \frac{1}{\,t\,} \cdot 
\text{Var} \big( \,h_1  \cdot   \mathbf{1}_{ \{ |h_1| \le \eps t \} }    \big)  \,= \, \sigma_t (\eps) - \frac{1}{\,t\,}\, \delta_t^2 (\eps)\,, \quad v_t := v_t (1)  
\label{6.9} \\
 \rho_t (\eps) \,&:= \,\frac{1}{\,t\,} \cdot \E  \Big( \big(\big| h_1 \big|  \wedge      \eps t \big)^2 \Big)  = \eps^2 \cdot \tau_t (\eps) + \sigma_t (\eps)\,, \quad \rho_t :=  \rho_t (1) 
 \label{6.8a}  
\end{align}
 for $\eps >0\,$, $N \in \N$, $t \in (0, \infty)$; all these are real numbers. We have then the celebrated  \textsc{Kolmogorov}--\textsc{Feller} (\cite{K1},\,\cite{F})   Weak Law of Large Numbers (cf.\,\cite{Kol};\,\cite{Ch},\,\S\,5.2;\,\cite{Du},\,\S\,2.2.3). The equivalence of    {\bf (i)},\,{\bf (ii)}  below is      in \textsc{Feller}\,\cite{F},\,pp.\,235-236;   the other   equivalent  conditions seem to be new.

\begin{theorem} 
  \label{Thm6.2}  {\bf Necessity and Sufficiency in the  Classical  WLLN for I.I.D. Sequences:}
  For   real-valued I.I.D. functions $\,h_1, h_2, \cdots,$  the following are equivalent: 
  
  \smallskip
  \noindent
  {\bf (i)} {\rm WLLN:} There exist ``correctors"     $\,d_1, d_2, \cdots\,$ in $\R\,,$ so that 
  $ \,
\lim_{N \to \infty} \left( \frac{1}{N} \sum_{n=1}^N h_n - d_N \right) = 0\,$ holds in $\,\P-$probability; i.e., for each $\eps >0$ we have
\begin{equation} 
\label{6.15}
\pi_N^\star( \eps)\,:= \,  \P \bigg( \, \bigg| \frac{1}{\,N\,} \sum_{n=1}^N h_n - d_N \bigg| > \eps  \bigg)\, \longrightarrow \, 0\,, \quad \text{as}~~ N \to \infty\,.
\end{equation}
 {\bf (ii)} With the notation of \eqref{6.7}, for every $\eps >0$ we have 
 \begin{equation} 
\label{6.16}
\lim_{M \to \infty} \tau_M (\eps) =0\,.
\end{equation}
{\bf (ii)$^\prime$} With the notation of \eqref{6.8}, for every $\eps >0$ we have 
 \begin{equation} 
\label{6.17}
\lim_{M \to \infty} \sigma_M (\eps) =0\,.
\end{equation}
{\bf (ii)$^{\prime \prime}$}  With the notation of \eqref{6.9}, for every $\eps >0$ we have 
 \begin{equation} 
\label{6.18}
\lim_{M \to \infty} v_M (\eps) =0\,.
\end{equation}
{\bf (ii)$^{\prime \prime \prime}$}  With the notation of \eqref{6.8a}, for every $\eps >0$ we have 
 \begin{equation} 
\label{6.17a}
\lim_{M \to \infty} \rho_M (\eps) =0\,.
\end{equation}
 {\bf (iii)} With the notation of \eqref{6.7},   for every $\eps >0$ we have 
 \begin{equation} 
\label{6.16a}
\lim_{M \to \infty} \frac{1}{\,M\,} \int_0^M \tau_t (\eps)\, \ud t\, =\,0\,.
\end{equation}
{\bf (iii)$^{\prime}$} With the notation of \eqref{6.8},   for every $\eps >0$ we have 
 \begin{equation} 
\label{6.16b}
\lim_{M \to \infty} \,M \int_M^\infty \frac{\sigma_t (\eps)}{\,t^2\,}  \, \ud t\, =\,0\,.
\end{equation}
{\bf (iii)$^{\prime \prime}$}  With the notation of \eqref{6.9},   for every $\eps >0$ we have 
 \begin{equation} 
\label{6.16c}
\lim_{M \to \infty} \,M \int_M^\infty \frac{v_t (\eps)}{\,t^2\,}  \, \ud t\, =\,0\,.
\end{equation}

When any $($therefore, all$\,)$ of  
\eqref{6.16}--\eqref{6.16c} hold, then   so does \eqref{6.15}   with $\, d_N \equiv \delta_N\,$  as in   \eqref{6.6}, and $\,\pi_N^\star( \eps) \equiv \pi_N ( \eps) $ as in  \eqref{6.9'}. 
   \end{theorem}

\subsection{The Exchangeable  Case}
\label{sec6.02}

 Let us consider now exchangeable, real-valued    $\,g_1, g_2, \cdots$ on the   probability space $\,(\Omega, \F, \P)$, with  
\begin{equation} 
\label{6.1}
{\cal T}_* := \bigcap_{n \in \N} \sigma \big( g_n, g_{n +1}, \cdots \big)
\end{equation}
  their tail $\sigma-$algebra (which is trivial  when the $g_1, g_2, \cdots$ are independent, by the \textsc{Kolmogorov} zero-one law).   Conditioned on this $\sigma-$algebra, the $\,g_1, g_2, \cdots$ become I.I.D.,  by the \textsc{de\,Finetti} theorem (e.g.\,\,\cite{CT}, p.\,222).  For $\eps >0$, $N \in \N$, $t \in (0, \infty)$  we introduce the conditional versions of the quantities in \eqref{6.6}--\eqref{6.9}, namely:  
\begin{align} 
\Delta_t (\eps) \, & := \,\E   \big( g_1 \cdot  \mathbf{1}_{ \{ |g_1| \le \eps t \} } \, \big| \, {\cal T}_* \big)\,, \quad \Delta_t := \Delta_t (1)
\label{6.2} \\
 \mathbold{\Pi}_N (\eps) \,&:= \,\P \bigg( \, \bigg| \frac{1}{\,N\,} \sum_{n=1}^N g_n - \Delta_N \bigg| > \eps \, \bigg| \, {\cal T}_* \bigg)\,, \quad  \mathbold{\Pi}_N :=  \mathbold{\Pi}_N (1)
 \label{6.5'} \\
p_N ( \eps) \,&:=\, \, \E \big[ \,
 \mathbold{\Pi}_N (\eps) \, \big] \,= \P \bigg( \, \bigg| \frac{1}{\,N\,} \sum_{n=1}^N g_n - \Delta_N \bigg| > \eps  \bigg), \quad p_N := p_N (1)
 \label{6.5''} \\
T_t (\eps) \,&:= \,t \cdot \P \big( \,  |g_1| > \eps\, t     \, \big| \, {\cal T}_* \big)\,, \quad T_t := T_t (1) 
\label{6.3} \\
 \Sigma_t (\eps) \,&:= \,\frac{1}{\,t\,} \cdot \E \big( \,g_1^2 \cdot   \mathbf{1}_{ \{ |g_1| \le \eps t \} }    \, \big| \, {\cal T}_* \big)\,, \quad \Sigma_t :=  \Sigma_t (1) 
 \label{6.4} \\
 V_t (\eps) \,&:= \,\frac{1}{\,t\,} \cdot \text{Var} \big( \,g_1  \cdot   \mathbf{1}_{ \{ |g_1| \le \eps t \} }    \, \big| \, {\cal T}_* \big)\,, \quad  V_t := V_t (1) 
\label{6.5} \\
  R_t (\eps) \,&:= \,\frac{1}{\,t\,} \cdot \E \Big[\, \Big( \,\big|g_1 \big|  \wedge      \eps t \Big)^2      \, \Big| \, {\cal T}_* \Big]= \eps^2 \cdot T_t (\eps) + \Sigma_t (\eps)\,, \quad  R_t :=  R_t (1) \,.
  \label{6.4a} 
\end{align}

It has   been an open question  for quite some time    (cf.\,\cite{St1} and its references), whether conditions both necessary and sufficient exist for the Weak Law of Large Numbers in the context of exchangeable sequences, just as they exist for    sequences of independent   functions with common distribution. 

 The following result   settles this question in the affirmative. 

 \begin{theorem} 
  \label{Thm6.1}  {\bf Necessity and Sufficiency in the WLLN for Exchangeable Sequences:}  For    real-valued, exchangeable functions $\,g_1, g_2, \cdots,$ the following are equivalent: 
  
  \smallskip
  \noindent
  {\bf (i)} {\rm WLLN:} There exist real-valued, ${\cal T}_*-$measurable ``correctors"   $\,D_1, D_2, \cdots\,$  such that, for each $\eps >0\,,$  
\begin{equation} 
\label{6.11'}
p_N^\dagger ( \eps) \,:=\,  \P \bigg( \, \bigg| \frac{1}{\,N\,} \sum_{n=1}^N g_n - D_N \bigg| > \eps  \bigg)\, \longrightarrow \, 0\,, \quad \text{as}~~ N \to \infty\,.
\end{equation}
{\bf  (i)$^\prime$} {\rm \textsc{Conditional} WLLN:} There exist  real-valued, ${\cal T}_*-$measurable ``correctors" $\,D_1, D_2, \cdots$  such that, for each $\eps >0\,,$ 
 \begin{equation} 
\label{6.11''}
 \mathbold{\Pi}_N^\dagger (\eps) \,:=\,  \P \bigg( \, \bigg| \frac{1}{\,N\,} \sum_{n=1}^N g_n - D_N \bigg| > \eps\, \bigg| \, {\cal T}_*  \bigg)\, \longrightarrow \, 0 \quad \text{as}~~ N \to \infty\,, \qquad \text{in $\,\P-$probability\,}.
\end{equation}
{\bf (ii)} With the notation of \eqref{6.3}, \eqref{6.4}, for every $\eps >0$ we have {\rm both} 
 \begin{equation} 
\label{6.12}
\lim_{t \to \infty} T_t (\eps) =0~~~\text{{\rm and}}~~ ~\lim_{t \to \infty} \Sigma_t (\eps) =0\,,\quad \text{in $\,\P-$probability\,}.
\end{equation}
\noindent
{\bf (ii)$^{\prime}$}   With the notation of \eqref{6.3},  \eqref{6.5}, for every $\eps >0$ we have   {\rm both} 
  \begin{equation} 
 \label{6.14}
 \lim_{t \to \infty} T_t (\eps) =0~~~\text{{\rm and}}~~ ~\lim_{t \to \infty} V_t (\eps) =0\,,\quad \text{  in  $\,\P-$probability\,}.
 \end{equation}
 {\bf (ii)$^{\prime \prime}$}   With the notation of    
\eqref{6.4a}, for every $\eps >0$ we have  
 \begin{equation} 
\label{6.12a}
\lim_{t \to \infty} R_t (\eps) =0\,,\quad \text{in $\,\P-$probability\,}.
\end{equation}
{\bf (iii)} With the notation of \eqref{6.3},   for every $\eps >0$ we have 
 \begin{equation} 
\label{6.16a}
\lim_{M \to \infty} \frac{1}{M} \int_0^M   T_t (\eps) \, \ud t \,=\,0\,,\quad \text{in $\,\P-$probability\,}.
\end{equation}
{\bf (iii)$^\prime$} With the notation of \eqref{6.4},   for every $\eps >0$ we have 
 \begin{equation} 
\label{6.16a}
\lim_{M \to \infty} M \int_M^\infty   \frac{\,\Sigma_t (\eps)\,}{t^2} \, \ud t\,=\,0\,,\quad \text{in $\,\P-$probability\,}.
\end{equation}
 {\bf (iii)$^{\prime \prime}$}  With the notation of \eqref{6.5},   for every $\eps >0$ we have 
 \begin{equation} 
\label{6.16a}
\lim_{M \to \infty} M \int_M^\infty   \frac{\,V_t (\eps)\,}{t^2} \, \ud t\,=\,0\,,\quad \text{in $\,\P-$probability\,}.
\end{equation}

When any (therefore, all) of the conditions  \eqref{6.12}--\eqref{6.16a}  hold, then so do {\bf (i)},\,{\bf  (i)$^\prime$}   with   $\, D_N \equiv \Delta_N\,$ \\ as  in \eqref{6.2}$;$ and then   $\,p_N^\dagger( \eps) \equiv p_N ( \eps) $ holds in \eqref{6.11'},\,\eqref{6.5''}, as does $\, \mathbold{\Pi}_N^\dagger (\eps) \equiv  \mathbold{\Pi}_N  (\eps)\,$ in \eqref{6.5'},\,\eqref{6.11''}. The  correctors  of \eqref{6.11'},\,\eqref{6.11''} satisfy  $\,  \lim_{N \to \infty} \big( D_N  /  \sqrt{N\,} \, \big) =0\,$ in $\P-$probability.

   \end{theorem}

\subsection{Some Counterexamples}
\label{sec6.0}

   Before presenting  the proofs of   (the essentially classical) Theorem \ref{Thm6.2}  and of (the novel) Theorem \ref{Thm6.1}, we discuss one   subtle difference:   {\sl ``\,There is no counterpart   in Theorem     \ref{Thm6.1} to the conditions {\bf (ii)},\,{\bf (ii)$^{\prime}$},\,{\bf (ii)$^{\prime \prime}$}  of Theorem \ref{Thm6.2}."}     Rather, these conditions {\it have to be combined,} as   they are in {\bf (ii)},\,{\bf (ii)$^{\prime}$},\,{\bf (ii)$^{\prime \prime}$}   of Theorem \ref{Thm6.1}, in order for their conditional versions to yield equivalent conditions in that context.

   \smallskip
 Here is   some intuition for this phenomenon. First,   \eqref{6.8a} gives the   
 identity $\, \rho_t = \sigma_t + \tau_t\,$. As we shall show in Lemma   \ref{Lem6.5} below, for I.I.D. $h_1, h_2, \cdots$ 
  {\it it is  indeed true} that $\, \big( \sigma_t \big)_{t >0 }\,$ remains bounded (respectively, tends to zero) if, and only if, $\, \big( \tau_t \big)_{t >0 }\,$ does so. 
  
     \smallskip
  But  there is a subtle caveat: {\it it may very well happen,   for an exchangeable sequence    $g_1, g_2, \cdots$ and for some  sequence  of natural numbers $N_1, N_2, \cdots$ increasing to infinity, that   we have }
  $$
  \lim_{k \to \infty} T_{N_k}   = 0 \qquad \text{while} \qquad  \lim_{k \to \infty} \Sigma_{N_k}   = \infty \,, 
  $$
{\it as well as }    
  $$
  \lim_{k \to \infty} T_{N_k -1}   = \infty \qquad \text{while} \qquad  \lim_{k \to \infty} \Sigma_{N_k -1}  = 0\,.
  $$
  Indeed, it suffices to consider  $g_1$ with   $\, \P ( g_1 = N_k ) = k / N_k\,$ for $k \in \N\,,$ and   zero otherwise; if    
  $\, \big( N_k \big)_{k \in \N} \subset \N\,$ tends to infinity sufficiently fast,   $g_1$ is well-defined  and both statements   above hold. 
  
  \smallskip
  Thus, the  conjunction of the two requirements in \eqref{6.12},    Theorem \ref{Thm6.1}\,(ii)$^\prime$  is necessary, as shown by    examples {\bf (${\mathfrak B}$)},\,{\bf (${\mathfrak C}$)}  of   Proposition   \ref{Prop6.3}.     This   will be proved in section \ref{sec6.44}, based on the above idea. 

\begin{proposition} 
  \label{Prop6.3}  {\bf Counterexamples:} The following hold. 
  
    \smallskip
  \noindent
  {\bf (${\mathfrak A}$)} There exist  exchangeable functions $\,g_1, g_2, \cdots$ with $($common$\,)$ symmetric distribution, for which

\smallskip
\noindent
$(a)$ the WLLN holds, namely, $\, \lim_{N \to \infty} \, \P \Big( \Big| \sum_{n=1}^N g_n \Big| > \eps \cdot N \Big) =0\,$ for every $~\eps >0\,;$

\smallskip
\noindent
$(a)^\prime$ the conditional $\,($in the a.e.\,sense, on the tail $\sigma-$algebra ${\cal T}_*)$ WLLN fails to hold on $(\Omega, {\cal T}_*, \P);$ namely, we have  $\, \overline{\lim}_{N \to \infty} \, \P \Big( \Big| \sum_{n=1}^N g_n \Big| >   C \cdot N \, \Big|\, {\cal T}_*\Big) >0\,,$    $~\P-$a.e., for every $C>0\,;$

\smallskip
\noindent
 $(b)$  $\,\overline{\lim}_{N \to \infty} \Sigma_N =1\,,$   $\,\P-$a.e.
 
 \smallskip
\noindent
 $(c)$  $\,\overline{\lim}_{N \to \infty}T_N =1\,,$   $\,\P-$a.e.

  \medskip
  \noindent
{\bf (${\mathfrak B}$)} There exist  exchangeable functions $\,g_1, g_2, \cdots$ with $($common$\,)$ symmetric distribution, for which

\smallskip
\noindent
$(a)$ the WLLN fails, namely, $\,\overline{\lim}_{N \to \infty} \, \P \Big( \Big| \sum_{n=1}^N g_n \Big| > C \cdot N \Big) =1\,$ for every $~C >0\,;$ 

\smallskip
\noindent
$(b)$ $\,\lim_{N \to \infty} \Sigma_N =0\,,$ \,in $\P-$probability; and

\smallskip
\noindent
$(c)$ $\,\overline{\lim}_{N \to \infty} \,\P \big( T_N \ge C \big) =1\,, ~~ \forall ~ C > 0\,.$

 \medskip
\noindent
{\bf (${\mathfrak C}$)} There exist  exchangeable functions $\,g_1, g_2, \cdots$ with $($common$\,)$ symmetric distribution, for which

\smallskip
\noindent
$(a)$ the WLLN fails, namely, $\,\overline{\lim}_{N \to \infty} \, \P \Big( \Big| \sum_{n=1}^N g_n \Big| > C \cdot N \Big) =1\,$ for every $~C >0\,;$

\smallskip
\noindent
$(b)$ $\,\lim_{N \to \infty}  T_N =0\,,$ \,in $\P-$probability; and

\smallskip
\noindent
$(c)$ $\,\overline{\lim}_{N \to \infty} \,\P \big( \Sigma_N \ge C \big) =1\,, ~~ \forall ~C > 0\,.$

\end{proposition}

\subsubsection{Consequences in the General Setting}
\label{sec3.1}

 Let us  consider an {\it arbitrary} sequence $f_1, f_2, \cdots$ of measurable functions on a probability space $(\Omega, \F, \P)$ and introduce, by analogy with the notation   in subsection \ref{sec6.02},   the ${\cal T}-$measurable   quantities
   \begin{equation} 
\label{D6}
 T^{\,\ell h}_t (\eps):=t \cdot  \sup_{n \in \N} \, \P \big( \,\big|f_{n} \big| > t \, \big| \, {\cal T}\, \big) , \qquad \Sigma^{\,\ell h}_t (\eps):= \frac{1}{\,t\,} \cdot \sup_{n \in \N} \, \E \Big( f_{n}^{\,2} \cdot \mathbf{1}_{\{ |f_{n}  |  \le \eps t \}} \, \Big| \, {\cal T}\, \Big),
\end{equation}
\newpage
\begin{equation} 
\label{D7}
\frac{1}{\,2\,} \, \Big( \eps^2 \cdot T^{\,\ell h}_t (\eps) + \Sigma^{\,\ell h}_t (\eps) \Big) \le  R^{\,\ell h}_t (\eps):= \frac{1}{\,t\,} \cdot \sup_{n \in \N} \, \E \Big( \big( \,\big|f_{n} \big|  \wedge    \eps t \, \big)^2 \, \Big| \, {\cal T}\, \Big) \le  \eps^2 \cdot T^{\,\ell h}_t (\eps) + \Sigma^{\,\ell h}_t (\eps)\,, 
\end{equation}
\begin{equation} 
\label{D8}
 V^{\,\ell h}_t (\eps) \,:= \,\frac{1}{\,t\,} \cdot \sup_{n \in \N} \, {\rm Var} \Big( \,\big|f_{n} \big|  \cdot   \mathbf{1}_{ \{  |f_{n}  | \le \eps t \} }    \, \Big| \, {\cal T} \, \Big) \,;
\end{equation}
 here   the superscript ``$\,\ell \,h$" denotes   ``\,lacunary/hereditary", and ${\cal T}$ is the tail-$\sigma$-algebra in \eqref{A6a}.

 \medskip
\noindent
$\bullet~$      We  {\bf posit} now the following conditions,  analogues of those in Theorem   \ref{Thm6.1} (conditions {\bf (i)},\,{\bf (i)$^\prime$} below are identical to their ``namesakes" in Theorem 3.1):

  \medskip
     \noindent 
     {\bf (i)} {\it Lacunary/Hereditary WLLN:} There exist $\,{\cal T}-$measurable ``correctors" $D_1, D_2, \cdots\,$ with $\,\P \big( \big| D_N \big| \le N,\, \forall \, N \in \N \big)=1\,$ and a subsequence $\,f_{k_1}, f_{k_2}, \cdots\,  $  along which,   and along whose every subsequence,  
     \begin{equation} 
\label{D2}
 \lim_{N \to \infty} \bigg(  \frac{1}{\,N\,} \sum_{n=1}^N f_{k_n}   - D_N \bigg) =\,0 \quad \text{holds in} ~\P\text{\,--\,probability.}    
\end{equation}
          {\bf (i)$^\prime$} {\it Lacunary/Hereditary Conditional WLLN:} There exist $\,{\cal T}-$measurable ``correctors" $D_1, D_2, \cdots\,$ with $\,\P \big( \big| D_N \big| \le N,\, \forall \, N \in \N \big)=1\,$ and a subsequence $\,f_{k_1}, f_{k_2}, \cdots\,  $  along which,   and along whose every subsequence, we have for every $\eps >0 \,,$
                     \begin{equation} 
\label{D5}
 \lim_{N \to \infty} \P  \bigg( \,  \bigg| \frac{1}{\,N\,} \sum_{n=1}^N f_{k_n}   - D_N \bigg| > \eps  \, \bigg| \,{\cal T} \bigg) =\,0 \quad \text{in} ~~\P\text{\,--\,probability.}    
\end{equation}
{\bf (ii)} With the notation of \eqref{D6},  for every $\eps >0$ we have {\rm both} 
 \begin{equation} 
\label{6.12z}
\lim_{t \to \infty} T^{\,\ell h}_t (\eps) =0~~~\text{{\rm and}}~~ ~\lim_{t \to \infty} \Sigma^{\,\ell h}_t (\eps) =0\,,\quad \text{in $\,\,\P-$probability\,}.
\end{equation}
{\bf (ii)$^{\prime}$}   With the notation of \eqref{D6},  \eqref{D7}, for every $\eps >0$ we have   {\rm both} 
  \begin{equation} 
 \label{6.14z}
 \lim_{t \to \infty} T^{\,\ell h}_t (\eps) =0~~~\text{{\rm and}}~~ ~\lim_{t \to \infty} V^{\,\ell h}_t (\eps) =0\,,\quad \text{  in  $\,\,\P-$probability}.
 \end{equation}
 {\bf (ii)$^{\prime \prime}$} With the notation of  
\eqref{D7}, for every $\eps >0$ we have  \begin{equation} 
\label{6.12az}
\lim_{t \to \infty} R^{\,\ell h}_t (\eps) =0\,,\quad \text{in $\,\P-$probability}.
\end{equation}
{\bf (iii)} With the notation of \eqref{D6},   for every $\eps >0$ we have 
 \begin{equation} 
\label{6.16az}
\lim_{M \to \infty} \frac{1}{M} \int_0^M   T^{\,\ell h}_t (\eps) \, \ud t \,=\,0\,,\quad \text{in $\,\P-$probability}.
\end{equation}
{\bf (iii)$^\prime$} With the notation of \eqref{D6},   for every $\eps >0$ we have 
 \begin{equation} 
\label{6.16az}
\lim_{M \to \infty} M \int_M^\infty   \frac{\,\Sigma^{\,\ell h}_t (\eps)\,}{t^2} \, \ud t\,=\,0\,,\quad \text{in $\,\P-$probability}.
\end{equation}
{\bf (iii)$^{\prime \prime}$}   With   the notation of \eqref{D8},   for every $\eps >0$ we have 
 \begin{equation} 
\label{6.16az}
\lim_{M \to \infty} M \int_M^\infty   \frac{\,V_t^{\,\ell h} (\eps)\,}{t^2} \, \ud t\,=\,0\,,\quad \text{in $\,\P-$probability}.
\end{equation}

\smallskip
It is very tempting  at this point,  to conjecture that all these conditions above are equivalent.    {\it    The counterexamples of   Proposition   \ref{Prop6.3}  demonstrate, however, that this conjecture is false.} While {\bf (i)} $ \Leftrightarrow$  {\bf (i)$^\prime$} still holds in the above list, Proposition \ref{Prop6.3}\,(${\mathfrak A}$) shows that {\bf (i)}  does not imply any of the conditions \eqref{6.12z}--\eqref{6.16az} of the above list. 

\smallskip
In the reverse direction,  however,  {\it each   of \eqref{6.12z}--\eqref{6.16az}  is strong enough to imply both {\bf (i)},\,{\bf (i)$^\prime$}.}   Indeed, after passing to a determining subsequence $\,f_{k_1}, f_{k_2}, \cdots\,$ of $\, f_1, f_2, \cdots\,$ with associated exchangeable    $g_1, g_2, \cdots\,$ as in (\ref{A7}) below, the assumptions  \eqref{6.12z}--\eqref{6.16az} on the $f_1, f_2, \cdots $ imply the corresponding properties for the sequence $g_1, g_2, \cdots\,$. Hence, by Theorem   \ref{Thm6.1}, the   $\,g_1, g_2, \cdots\,$ satisfy the WLLN, thus also condition {\bf (ii)}  of Theorem  \ref{Thm3.2}. This shows the following. 

\begin{corollary} 
{\bf Sufficient Conditions:} 
Any one of the conditions  \eqref{6.12z}--\,\eqref{6.16az} implies both the Lacunary/Hereditary WLLN \eqref{D2},  and the Lacunary/Hereditary Conditional WLLN \eqref{D5}.  
\end{corollary}

\section{The Proof  of Theorem  \ref{Thm2.1} 
}
\label{sec4}

We start with a  straightforward observation.

 \begin{lemma} 
  \label{Lem4.1}
  Suppose that the real-valued, measurable   $\,f_1, f_2, \cdots$ contain  a subsequence   $\,f_{k_1}, f_{k_2}, \cdots\,$ along which,     and along whose every    further  subsequence, \eqref{D2a} holds for some real-valued   $\,D_1, D_2, \cdots\,$     measurable with respect to the tail $\sigma-$algebra ${\cal T}$ of \eqref{A6a}. 
  
  Then we may assume $($passing to a further  subsequence$)$ that $\,f_{k_1}, f_{k_2}, \cdots\,$   is bounded in probability.
   \end{lemma}

   \noindent
       {\it Proof:}                  Assume the contrary; i.e., that $\,f_{k_1}, f_{k_2}, \cdots\,$ is {\it not} bounded in probability. Then  
\begin{equation} 
\label{A4}
\lim_{M \to \infty} \, \sup_{n \in \N} \,\P \,\big( \big|f_{k_n} \big| > M \big) > \alpha
\end{equation}
 holds for some $\alpha >0\,;$ and we   construct, inductively, a subsequence $f_{k_{n_1}}, f_{k_{n_2}}, \cdots\,$ with the property 
 \begin{equation} 
\label{A5}
  \P \bigg(\, \bigg| \,\frac{1}{\,L\,} \sum_{\ell=1}^L f_{k_{n_\ell}} - D_L\, \bigg| >1 \bigg) \, > \, \frac{\, \alpha\,}{2}\,, \qquad \forall~~ L \in \N\,.
\end{equation}  
This  will provide  the desired contradiction to the hereditary validity of \eqref{D2}.

Indeed, suppose that the terms $\,f_{k_{n_1}}, \cdots, f_{k_{n_L}}  \,$ have been chosen   to satisfy \eqref{A5}; selecting then $M>2\,$  in \eqref{A4} large enough,   we determine the index $\,k_{\,n_{L+1}}>k_{\,n_{L}}\,$ of the next term,  so that
$$
\P \bigg( \bigg| \frac{1}{L+1} \sum_{\ell=1}^{L+1} f_{k_{n_\ell}} - D_{L+1} \bigg| >\frac{\,M\,}{2} \bigg) \,  > \, \P \bigg( \frac{\,\big|f_{k_{\,n_{L+1}}} \big| \,}{L+1} > M   \bigg)  - \, \P \bigg( \bigg| \frac{1}{L+1} \sum_{\ell=1}^{L } f_{k_{n_\ell}} - D_{L+1} \bigg| >\frac{\,M\,}{2} \bigg)
 > \, \frac{\alpha }{\,2\,} 
$$
holds. Here, the  last displayed probability  is less than $\,\alpha/ 2\,$ for $M>2$ sufficiently large, because of  \eqref{D2}; whereas, after fixing such $M$, we can choose $\,k_{\,n_{L+1}}>k_{\,n_{L}}\,$ so that the next-to-last displayed   probability is bigger than $\alpha$, on the strength of  \eqref{A4}. 

Recalling   $M>2\,,$ we see that the inductive step   $ L \longmapsto L+1$   for \eqref{A5} has been established.    \qed

\subsection{Proof of Theorem  \ref{Thm2.1}\,(i)}
\label{4.1a}

This is a consequence of  the   \textsc{Koml\'os} \cite{K} result,  and of the fact that  the \textsc{Ces\`aro} limit $f_*$ can be defined  on the increasing sequence of sets $A_1, A_2, \cdots$ consistently, via diagonalization,  and in the spirit of the ``\textsc{Chacon} Biting Lemma" in \cite{BC1}.

More specifically, from the   \textsc{Koml\'os} theorem there exist for each $j \in \N$ a measurable   $f^{(j)}_* : \Omega \to \R$ and a subsequence $f^{(j)}_{k_1}, f^{(j)}_{k_2}, \cdots\,$ of $f_{k_1}, f_{k_2}, \cdots\,$, so that $\,f^{(j)}_*=\lim_{N \to \infty} \frac{1}{N} \sum_{n=1}^N f_{k_n}^{(j)}\, \mathbf{ 1}_{A_{j}}     \,$ holds $\,\P-$a.e. Clearly,   the $f^{(j+1)}_{k_1}, f^{(j+1)}_{k_2}, \cdots\,$ can be selected as a subsequence of $f^{(j)}_{k_1}, f^{(j)}_{k_2}, \cdots\,;$ and   $f^{(j+1)}_* = f^{(j)}_*$ holds $\P-$a.e.\,on $A_j$. We define now, consistently, 
$$
 f_* (\omega) :=  f^{(j)}_* (\omega)\,, ~~~ \omega \in A_j\,; \qquad  f_* (\omega) := 0\,, ~~~ \omega \in \Omega \setminus A
$$
with $A := \bigcup_{j \in \N} A_j\,$ a set of full measure $\P(A)=1,$ and notice that for each fixed  $j \in \N$ we have
$$
 f_* (\omega)  =  f^{(j)}_* (\omega)=\lim_{N \to \infty} \frac{1}{N} \sum_{n=1}^N f_{k_n}^{(j)}(\omega)=\lim_{N \to \infty} \frac{1}{N} \left(\sum_{n=1}^j f_{k_n}^{(j)}(\omega) +\sum_{n=j+1}^N f_{k_n}^{(j)}(\omega) \right)
$$
\begin{equation} 
\label{A15}
~~~~~~~~~~~~~~~~~=\lim_{N \to \infty} \frac{1}{N}  \sum_{n=j+1}^N f_{k_n}^{(n)}(\omega) =\lim_{N \to \infty} \frac{1}{N}  \sum_{n= 1}^N f_{k_n}^{(n)}(\omega)\,, \qquad \omega \in A_j\,.
\end{equation} 
The fourth equality  recalls that $\, \big( f_{k_n}^{(n)} \big)_{n>j}\,$ is a subsequence of $\, \big( f_{k_n}^{(j)} \big)_{n \in \N}\,$, and the hereditary nature of the \textsc{Ces\`aro}-mean convergence in the  \textsc{Koml\'os} theorem.   The first and last terms  of \eqref{A15} do not involve $j$,  thus agree on the set of full measure $A= \bigcup_{j \in \N} A_j\,$. This is the   \textsc{Ces\`aro}-mean convergence along the ``diagonal" subsequence $\, \big( f_{k_n}^{(n)} \big)_{n \in \N}\,;$ its hereditary character  is clear.  \qed

\subsection{Proof of Theorem \ref{Thm2.1}\,(ii)}
\label{4.1b}

\noindent
{\bf STEP 1: Determining Subsequence.} On the strength of the     boundedness-in-probability established in Lemma \ref{Lem4.1}, we may assume      (passing to a further subsequence, if needed)  the   subsequence   $f_{k_1}, f_{k_2}, \cdots\,$ to be   ``determining"    in the sense,  and with the notation,  of    \eqref{B0}, \eqref{A6a}.

   \smallskip
\noindent
{\bf STEP 2: Approximation by an Exchangeable Sequence.}    Following the trail blazed by \textsc{Aldous} \cite{A}-\cite{AE} (cf.\,\textsc{Dacunha-Castelle}\,\cite{DC};\,\cite{BerPet}), we summon now an {\it exchangeable}  sequence of   functions $\,  g_1, g_2, \cdots\,,$ conditionally   independent and with the    random probability  measure $\omega \mapsto {\bm \mu}_\omega$\,(cf.\,\eqref{B0} and the paragraph below it)   as their common conditional distribution    given  the tail-$\sigma$-algebra 
\begin{equation} 
\label{A11}
{\cal T}_*  : = \bigcap_{n \in \N}   {\bm \sigma} \big(  g_{{n}},  g_{{n+1}} , \cdots \big) \, \subseteq  \, {\cal T}\,;
\end{equation}
 and also such that, passing to a further subsequence of the determining $\,f_{k_1}, f_{k_2},\cdots\,$  if necessary, 
 \begin{equation} 
\label{A7}
 \big( f_{k_{n+1}}, \cdots, f_{k_{n+M}} \big) ~~ \text{converges in distribution as $ n \to \infty\,$ to} ~~ \big( g_1, \cdots, g_M \big)~~
\end{equation}
 for all $M \in \N\,$ and     $\,M-$tuples $(k_{n+1} , \cdots, k_{n+M})$  with $\,k_{n+1} < \cdots < k_{n+M}\,,$   $\, \lim_{n \to \infty} k_{n+1} = \infty\,.$   
 
  It follows now from Proposition \ref{Prop3.1}(ii),      that 
  the assumed hereditary $\P-$a.e.\,validity of \eqref{2}, for some  
   $f_* \in \Ll^0\,,$ implies    
   the $\P-$a.e.\,finiteness  
\begin{equation} 
\label{A9}
\P \big(  {\bm \kappa} < \infty \big) =1\quad \text{for the function} \quad \Omega \ni \omega \, \longmapsto \, {\bm \kappa} (\omega) := \int_\R \big| x \big| \, {\bm \mu}_\omega ( \ud x   ) \in [0, \infty]\,.
\end{equation}    
{\bf STEP 3: The condition (1.3).} On the strength of \eqref{A9},  we find for any given    $ \varepsilon \in (0,1)$ some $  K \in (0, \infty)$ such that the set $B:= \{ {\bm \kappa}  \le K \} \in  {\cal T}$ has measure $ \P (B)   > 1 - \varepsilon.$     The sequence $f_{k_1}, f_{k_2}, \cdots $ is determining, so        
$\,f_{k_1}\, \mathbf{1}_{B}\,, \,f_{k_2}\, \mathbf{1}_{B}\,, \cdots\,$ converges in distribution to a random variable $\xi$ with   distribution ${\bm \nu}$ and first absolute  moment bounded by $K$ 
(this   holds even conditionally on ${\cal T}$).   

 Then, under very mild   conditions on the underlying space  $(\Omega, \F),$ which we can certainly assume here, there exist a (relabelled) subsequence $f_{k_1}, f_{k_2}, \cdots\,,$ as well as a sequence of random variables $\xi_1, \xi_2, \cdots\,$ with common  distribution ${\bm \nu}\,,$ such that $ \, \lim_{n \to \infty} \big( f_{k_n} \mathbf{1}_{B} - \xi_n \big)=0\, $ holds $\P-$a.e.\,(in the spirit of the \textsc{Skorokhod}, a.k.a.\,\,\textsc{Strassen-Dudley}, construction in\,\cite{Bill},\,\,p.\,70; a more direct reference is Lemma 4, p.\,41 of \cite{ABT}). But now \textsc{Egorov}'s theorem gives  a subset   $A \in {\cal T}$ of $B$ with $ \P (A) > 1 - 2 \varepsilon   $ and a  (further, relabelled) subsequence with $  \xi $   bounded on $A$, $\, \sup_{\omega \in A} \big| f_{k_n} (\omega) - \xi_n (\omega) \big| \to 0 \,$ as $n \to \infty\,,$ and 
 $\,
 \sup_{n \in \N}\, \big\| \, f_{k_n} \mathbf{1}_{A}\,\big\|_{\Ll^1 (\P)} \,<\, K+1\,.
 $

 \smallskip
Iterating this procedure,  we  come up with   
real numbers $\, 0 < K_1 <K_2< \cdots < K_j \uparrow \infty\,$    and    \textsc{Egorov}-type sets $\, A_1 \subseteq A_2 \subseteq \cdots \subseteq A_j \subseteq \cdots \,$  in ${\cal T},$  such that:  \\ 
  $(i)$ ~for each $j \in \N$,  we have     the bounds 
  $\,
\P \big( A_j \big) \ge 1 - 2^{\,-j}\,,$      $\, {\bm \kappa} (\omega) \le K_j ~ \text{for $\,\P-$a.e.}~ \omega \in A_j  \,  ;
  $
 and  \\ $(ii)$\, the  sequence  $\, \big(   f_{k_{n }}   \cdot \mathbf{ 1}_{ A_j} \big)_{n \in \N} \,$ is bounded in $\, \Ll^1 (\P)\,,$ as claimed in \eqref{A1}. \qed

 \smallskip
     \begin{remark} 
  {\rm
The inclusion $\,{\cal T}_*    \subseteq    {\cal T}\,$ in \eqref{A11} follows from the study of ``almost exchangeable"  sequences   in   \cite{BerRos}, most notably its    Proposition 2.3.   {\it  This inclusion can, very easily, be strict.}

Indeed, consider the sequence of uniformly bounded functions $\, f_n := (1/n) \cdot \mathbf{1}_{[0,1/2]}\,, \, \,n \in \N\,$ on $([0,1], {\cal B}, {\bm \lambda})$. This sequence is determining, and the associated exchangeable functions $g_1 \equiv 0,\, g_2 \equiv 0,\cdots$  have trivial tail $\sigma-$algebra ${\cal T}_*\,$; yet the  tail $\sigma-$algebra ${\cal T}$  of the $f_1, f_2, \cdots$   contains the interval (1/2, 1]. All   results below continue to hold, if conditioning on ${\cal T}_*\,$ is replaced by conditioning on ${\cal T}\,.$
}
    \end{remark}

\subsection{Proof of Theorem \ref{Thm2.1}\,(iii)}
\label{4.1c}

  Assuming the validity of \eqref{A1} for some subsequence $f_{k_1}, f_{k_2}, \cdots\,$   and    sets $\,A_1 \subseteq A_2 \subseteq \cdots \subseteq A_j \subseteq \cdots\,$ in $\,\F$ with $\,\lim_{j \to \infty} \P (A_j ) =1\,,$ we define a new probability measure $\,\Q \sim \P\,$ via
\begin{equation} 
\label{A10}
 \frac{\ud \Q\,}{\ud \P} \, \big( \omega   \big) \,:= \, \bigg( \sum_{j \in \N} \alpha_j \, \mathbf{1}_{A_{j}} ( \omega) \bigg) \bigg/ \,   \sum_{j \in \N} \alpha_j \, \P (A_{j})  \,,
\end{equation}
where the numbers $\, \alpha_j \in (0,1)\,$ satisfy  $ \,  \alpha_j \,\cdot \, \sup_{n \in \N} \big\| \big( 1 + \big|f_{k_n} \big| \big) \, \mathbf{1}_{A_{j}} \big\|_{\Ll^1 (\P)}  \le (1/ j^2)\,, ~ \forall  ~ j \in \N\,;$    thus,    $\, \sup_{n \in \N} \big\| f_{k_n}   \big\|_{\Ll^1 (\Q)} < \infty\,$ follows. Whereas,  if   $\, \sup_{n \in \N} \big\| f_{k_n}   \big\|_{\Ll^1 (\Q)} < \infty\,$ holds for some probability measure $\,\Q \sim \P\,,$ then so does   \eqref{A1}   for the sets $\, A_j := \big\{  \big( \ud \P /  \ud \Q \big) \le j\big\}\,,~~ j \in \N\,.$ \qed

\section{Linking the Lacunary/Hereditary Case with Exchangeability}
\label{sec5}

We state  now and prove a   result    of  considerable  independent interest,    crucial for proving   Theorems   \ref{Thm2.1}, \ref{Thm3.2}. It comes in two versions, one for the Weak,   the other for the  Strong, Law of Large Numbers.

  \begin{proposition} 
  \label{Prop3.1}
  On a given probability space $\,(\Omega, \F, \P)$, consider a sequence of real-valued, measurable functions $f_1, f_2, \cdots$   containing a    subsequence  $\,f_{k_1}, f_{k_2}, \cdots \,$ with the  stable convergence    \eqref{B0}.   

\smallskip
\noindent
{\bf (i)}  Suppose  that   $\,\lim_{N \to \infty} \big(  \frac{1}{\,N\,} \sum_{n=1}^N f_{k_n}   - D_N \big) =\,0 \,~ \text{holds in} ~\P\text{\,--\,probability}$ for {\rm some}   $\,{\cal T}-$measurable correctors    $\,D_1, D_2, \cdots\,,$ along this subsequence  $\,f_{k_1}, f_{k_2}, \cdots\,$ and    all its subsequences.  Then,    the approximating exchangeable  functions    $\,g_1, g_2, \cdots\,$  of \eqref{A7}      also satisfy
 \begin{equation} 
 \label{6.10}
   \lim_{N \to \infty} \bigg(  \frac{1}{\,N\,} \sum_{n=1}^N g_{n}   -   D_N \bigg) =\,0 \quad \text{in} ~\P\text{\,--\,probability}  \end{equation}
   with  the  same  $\,{\cal T}-$measurable      $\,D_1, D_2, \cdots\,,$ as well as with the ${\cal T}_*-$measurable $\, \Delta_1, \Delta_2, \cdots\,$  of \eqref{6.2}.

\smallskip
\noindent
{\bf (ii)}  In the same setting,   suppose that the subsequence $\,f_{k_1}, f_{k_2}, \cdots,$  and all its subsequences,  satisfies  $\,\lim_{N \to \infty}    \frac{1}{\,N\,} \sum_{n=1}^N f_{k_n}    = 0 \,,~   \P-\text{a.e.}$ Then the   exchangeable    $\,g_1, g_2, \cdots\,$  of \eqref{A7}     also have this property:   
 \begin{equation} 
 \label{D2bb}
  \lim_{N \to \infty}    \frac{1}{\,N\,} \sum_{n=1}^N g_{n}  \,   =\,0 \,,~\quad  \P-\text{a.e.}    
 \end{equation} 
    \end{proposition}

\subsection{Proof of Proposition  \ref{Prop3.1}(i)}
\label{5.1}

To  alleviate notation somewhat, let us   write $\, f_1, f_2, \cdots\,$ for the sequence $\, f_{k_1}, f_{k_2}, \cdots\,$. We shall argue by contradiction: If     \eqref{6.10} fails,       there exist       $\, \alpha \in (0,1)\,$ and a  sequence      $\,N_1, N_2, \cdots$ in $\N$  with 
\begin{equation} 
\label{B3a}
\P \,\bigg( \bigg| \sum_{n=1}^{N_j} g_n - N_j \cdot D_{N_j} \bigg| > \alpha \, N_j \bigg) \, \ge \, \alpha\,,  ~~~~ \forall ~~ j \in \N\,.
\end{equation}
 Passing inductively to a (relabelled) subsequence of this $\,N_1, N_2, \cdots$ we may suppose that, for each $\ell \in \N\,$ and already defined $\,N_1, N_2, \cdots, N_\ell \, ,$ the next term $\,N_{\ell +1}\,$ is chosen big enough, so that, with $\,M_\ell \,:=\, N_1+\cdots + N_\ell\,, $  
$$
\P \,\Big( \,\big| F   \big| > \frac{\alpha}{\,4\,} \,   N_{ \ell+1}  \Big) \, < \,   \frac{\alpha}{\,4\,}
$$ 
holds for any   sum $\,F$    of at most $\, M_\ell\,-\,$many terms   from the sequences $\,f_{ 1}, f_{ 2}, \cdots\,$ or $\,g_1, g_2, \cdots\,.$  

Next, we choose a sequence $\,K_1, K_2, \cdots\,$ in $\N\,$ which  converges  to infinity sufficiently fast, and  is   such that for every $\,j \in \N\,$ we have 
$\,
K_j + N_j \,<\, K_{j+1}\,    
\,$
as well as 
\begin{equation} 
\label{B4a}
{\bm \varrho}_j \, \Big( \text{distribution} \big( f_{K_j +1}  , \cdots , f_{K_j+N_j }\big), \, \text{distribution} \big( g_1, \cdots, g_{N_j} \big) \Big) \, < \, \varepsilon_j  \,.
\end{equation}
Here    ${\bm \varrho}_j \, $ denotes  the \textsc{Prokhorov} distance (\cite{Bill},\,\,p.\,72) between probability measures on the \textsc{Borel} sets of $\,\R^{N_j}$, and   $\, \varepsilon_j \in (0,1)\,$ is small enough  so   we have also
\begin{equation} 
\label{B4b}
\P \bigg ( \, \bigg| \sum_{n=1}^{N_j} f_{K_j +n} - \sum_{n=1}^{N_j} g_{  n} \bigg| > \frac{\alpha}{\,4\,} \bigg)\,<\, \frac{\alpha}{\,4\,}\,, \qquad \forall ~~ j \in \N\,.
\end{equation}

\smallskip
We denote now by $\,f_{k_1}, f_{k_2}, \cdots\,$ the subsequence of $\, f_1, f_2, \cdots\,$ obtained by gluing together   consecutive blocks
   $\,   {\cal I}_j \,:=\, \big\{ K_j + 1, \cdots, K_{j}+ N_j   \big\} \,, ~ j \in \N\, 
 $ of indices and, for each $L \in N\,,$ estimate 
$$
\P \,\bigg( \,\bigg| \sum_{n =1}^{M_{L+1}} 
f_{\,k_n} - M_{L+1} \cdot D_{M_{L+1}} \bigg| > \frac{\alpha}{\,4\,} \, M_{L+1}  \bigg)  \ge \,\P \,\bigg( \,\bigg| \sum_{n =M_L +1}^{M_{L+1}} 
f_{\,k_n} - M_{L+1} \cdot D_{M_{L+1}} \bigg| > \frac{2\,\alpha}{\,4\,} \, M_{L+1}  \bigg) - \frac{\alpha}{\,4\,}
$$
$$
\ge \,\P \,\bigg( \,\bigg| \sum_{n =M_L +1}^{M_{L+1}} 
g_{ n} - M_{L+1} \cdot D_{M_{L+1}} \bigg| > \frac{3\,\alpha}{\,4\,} \, M_{L+1}  \bigg) - \frac{2\,\alpha}{\,4\,}\,
 $$
 $$
~~\ge \,\P \,\bigg( \,\bigg| \sum_{n  =1}^{M_{L+1}} 
g_{ n} - M_{L+1} \cdot D_{M_{L+1}} \bigg| >  \alpha  \, M_{L+1}  \bigg) - \frac{3\,\alpha}{\,4\,}\, \ge \, \frac{\alpha}{\,4\,}\,;
$$
here, the second inequality holds in light of  \eqref{A7},\,\eqref{B4a}--\eqref{B4b}, and the last in light of \eqref{B3a}. But  the resulting subsequence of $\,f_{k_1}, f_{k_2}, \cdots\,$   fails then to satisfy \eqref{D2a}, contrary to the premise  of Proposition   \ref{Prop3.1}(i) and leading to the desired contradiction. 

 The Weak Law of Large Numbers \eqref{6.10} is thus established.  \qed

\subsection{Proof of Proposition  \ref{Prop3.1}(ii)}
\label{5.2}

We argue again by contradiction,     assuming  that the approximating exchangeable functions $g_1, g_2, \cdots$ of \eqref{A7} fail to  converge in \textsc{Ces\`aro} mean  to $f_* \equiv 0\,,$  $\,\P-$a.e.;    there exists then $\, \alpha \in (0,1)\,$ with 
\begin{equation} 
\label{B9}
\P \big(B\big) \ge \alpha\,, \quad \text{for} ~\quad B\, := \bigcap_{M \in \N} ~ \bigcup_{N \ge M , \,N \in \N} B_N^{(g)}\,, ~~~~ B_N^{(g)} := \bigg\{ \bigg| \sum_{n=1}^{N } g_n \bigg| > \alpha \, N  \bigg\}\,.
\end{equation}
We proceed   as in   \S\,\ref{5.1}: consider   sequences of  natural numbers $\, N_1, N_2,    \cdots $ and  $\, K_1, K_2,   \cdots $ with  $\,K_j + N_j \,<\, K_{j+1}\,,   \,$ so that the terms over the disjoint blocks of indices $\,   {\cal I}_j \,:=\, \big\{ K_j + 1, \cdots, K_{j}+ N_j   \big\} \,, ~ j \in \N\, 
 $   are ``glued together" in a (relabelled) subsequence $\,f_{k_1}, f_{k_2}, \cdots\,$. 
 
 \smallskip
 There is now the following additional feature: we choose also for each $j \in \N$ a (large) number $\,C_j \ll N_j\,$ and  ensure that, in each inductive step $\, L \mapsto L+1\,,$ the following holds for these choices: Suppose $K_L, C_L, N_L\,$ have been chosen, and let   $\,
M_L \,:=\, N_1+\cdots + N_L\,;
 $
choose $C_{L+1}$  first, so that 
 $$
 \P \bigg ( \, \bigg| \sum_{n=1}^{M_L} f_{ k_n}  \bigg| > \frac{\alpha}{\,4\,} \, C_{L+1}\bigg)\,<\, \frac{\alpha}{\,4\,}\,,\qquad  \P \bigg ( \, \bigg| \sum_{n=1}^{M_L} g_{ n}  \bigg| > \frac{\alpha}{\,4\,} \, C_{L+1}\bigg)\,<\, \frac{\alpha}{\,4\,}
 $$
hold, then choose $N_{L+1} > C_{L+1}$ so that, in the notation of 
 \eqref{B9} for the sets defined there, we have 
 $$
\P \,\bigg( \, \bigcup_{N = M_L + C_{L+1}}^{N_{L+1}} B_N^{(g)}  \bigg) \,>\, \alpha\,.
 $$
  Finally, we choose $K_L$ large enough, so that the distributions of $\, \big( f_{K_j +1}, \cdots, f_{K_j + N_j} \big) \,$  and $\, \big( g_{ 1}, \cdots, g_{  N_j} \big) \,$ are close enough, in such a manner that the probability of the set 
 $$
 \bigcup_{N = M_L + C_{L+1}}^{N_{L+1}} B_N^{(g)}\quad \text{is close to that of} \quad \bigcup_{N = M_L + C_{L+1}}^{N_{L+1}} B_N^{(f)}\,,\qquad \text{where} \quad B_N^{(f)} := \bigg\{ \bigg| \sum_{n=1}^{N } f_{K_j +n} \bigg| > \alpha \, N  \bigg\}\,.
 $$
 We   follow then the chain of inequalities at the end of  \S\,\ref{5.1}, and obtain the desired contradiction. 
 
 The Strong Law of Large Numbers \eqref{D2bb} is thus established.  \qed

    \begin{remark} 
{\rm
In the context of Proposition   \ref{Prop3.1}(ii), the exchangeable functions $\,g_1, g_2, \cdots$ of \eqref{A7},    conditionally independent and with  common 
    distribution $\,{\bm \mu}_{\,\omega} \equiv {\bm \mu} (\omega) $  given ${\bm \sigma} ({\bm \mu})$,    converge     in \textsc{Ces\`aro} mean  to zero, \,{\it for $\,\P-$a.e.\,$\,\omega \in \Omega$} (i.e.,  \eqref{D2bb} holds).  From  the  converse  to the  Strong Law of Large Numbers (\textsc{Lo\`eve}\,\cite{Loe}, p.\,251), this common conditional   distribution    satisfies  then  
     \begin{equation} 
\label{B2}
\E \big[ \,\big| g_1 \big| \, \big| \, {\cal T}_* \, \big] (\omega) = 
 \int_\R \big| x \big| \,\, {\bm \mu}_\omega \big( \ud x  \big)
  < \infty\,\qquad \text{for}~~ \P-\text{a.e.}~  \omega \in \Omega\,.
\end{equation}
 \noindent
(The corresponding statement for the Weak Law of Large Numbers, i.e., in the context of Proposition   \ref{Prop3.1}(i), is developed in subsection \ref{sec10a}.)

 Conversely,  \eqref{B2}   implies the \textsc{Ces\`aro}-convergence  \eqref{D2bb} of the   $g_1, g_2, \cdots$   to zero, $\P-$a.e.; as well as this same convergence    for the {\rm ``determining"}  subsequence  $f_{k_1}, f_{k_2}, \cdots\,$    and for all its subsequences. Compare with  the  general formulation  of the  hereditary  SLLN in Theorem 3 in  \textsc{Aldous}\,\cite{A}.  }
      \end{remark}

\section{The Proof  of Theorem  \ref{Thm6.2}  }
\label{sec6.1}
 
 We start with a couple of preliminary results pertaining to the setting of Theorem  \ref{Thm6.2}.
 
  \begin{lemma} 
  \label{Lem6.4}
For any given $\eta >0$ there exists an integer $N_\eta \in \N$ such that, in the notation of \eqref{6.6}, \eqref{6.8}, we have
\begin{equation} 
\label{6.19}
\delta_N^2\, \le \, \eta \, N \, \big( 1 + \sigma_N \big)\,, \qquad ~~\forall~~ N \ge N_\eta\,.
\end{equation}
   \end{lemma} 
   
   \noindent
   {\it Proof:} We select $p < \eta / 2\,,$ and $M>0$ big enough so that $\P \big( |h_1| > M \big) < p\,.$ Then we have 
$$
\E \big( h_1^2 \cdot \mathbf{1}_{\{ |h_1| \le M \}} \big) \le M^2 \le \big( \eta / 2 \big) \,N\qquad \text{as long as} ~~ N>M\,,~ \eta\, N > 2\, M^2\,.
$$
In particular, with $\, h^{(M)}:= h_1\cdot  \mathbf{1}_{\{ |h_1| \le M \}}\,$, we have $\, \E \big[\, \big( h^{(M)} \big)^2\, \big] \le \big( \eta / 2 \big) \,N\,$.

\smallskip
We note now that the difference  $\, h^{(N)} - h^{(M)} = h_1\cdot  \mathbf{1}_{\{ M < |h_1| \le N \}}\,$  is supported on a subset of $\, \{ \,|h_1| > M \, \}\,,$ whose $\P-$measure does not exceed $p$. Therefore, by \textsc{Cauchy-Schwarz} we have 
$$
\big| \,\E \big( h^{(N)} - h^{(M)} \big) \, \big|^2 \,\le \, 
\Big( \,\E \big( \, \big| h^{(N)} - h^{(M)}  \big|\,\big) \, \Big)^2 \,\le \,  \E \bigg( \Big( \big| h_1 \big|\cdot  \mathbf{1}_{\{ M < |h_1| \le N \} } \Big) \bigg)^2
$$
$$
~~~~~~~~~~~~~~~\le \, \E \Big( h_1^2 \cdot \mathbf{1}_{\{  |h_1| \le N \} } \Big) \cdot \P \big( |h_1| > M \big) \,\le\, N \, \sigma_N \cdot p\,.
$$
Therefore, 
$ 
\,\delta_N = \E \big[ \, h^{(N)} \,\big] = \E \big[ \, \big( h^{(N)}- h^{(M)} \big)+ h^{(M)} \,\big] 
\,$
satisfies 
$$
\delta_N^2 \, \le \, 2 \cdot \E \big[ \, \big( h^{(N)}- h^{(M)} \big)^2  \,\big] + 2 \cdot \E \big[ \, \big(  h^{(M)} \big)^2  \,\big]\,\le\, 2\,p\,N\, \sigma_N + \eta \, N \, < \, \eta \, N \, \big( 1 + \sigma_N \big)\,.~~~ \qed
$$
 
      \smallskip
    \begin{lemma} 
  \label{Lem6.5}
In the notation of \eqref{6.7}, \eqref{6.8}, we have 
\begin{equation} 
\label{6.20}
\lim_{M \to \infty} \tau_M = 0 ~ \Longleftrightarrow ~ \lim_{M \to \infty} \sigma_M = 0 \,.
\end{equation}
   \end{lemma}

\noindent
{\it Proof:} We recall  the quantities of \eqref{6.7},\,\eqref{6.8} as 
$ 
\tau (t)   = t \cdot \P \big( |h_1| > t \big)\,, ~ \sigma (t)  = \frac{1}{\,t\,} \cdot \E \big( h_1^2 \cdot \mathbf{1}_{ \{ |h_1| \le t \} } \big)
 $
for $ t \in (0, \infty)\,;$   setting $\tau (0) := \tau (0+) =0\,,$ $\sigma (0) := \sigma (0+) =0\,,$ we obtain   functions defined and right-continuous on $[0,\infty)\,,$ as well as bounded on compact intervals.  Then integration by parts gives   
\begin{equation} 
\label{6.22}
\tau (M) = \frac{2}{\,M\,} \int_0^M \tau (t) \, \ud t - \sigma(M)\,, \qquad M > 0
\end{equation}
as in (7.7) of p.\,235 in \textsc{Feller},\,Vol.\,II \cite{F} (after correcting for a typo there), which leads to the implication $\,\lim_{M \to \infty} \tau (M)  = 0 ~ \,\Longrightarrow ~ \lim_{M \to \infty} \sigma (M)  = 0\,.$

 \smallskip

Now let us look at \eqref{6.22} as an integral equation for the function $\tau (\cdot)$, in terms of some ``given" function $\sigma (\cdot)$. We assume for a moment that both $\tau (\cdot)\,$ and $\sigma (\cdot)\,$ are continuous and continuously differentiable, and obtain from \eqref{6.22} by differentiation $\, \tau (t) - t  \,\tau^\prime (t)= \sigma (t) + t  \,\sigma^\prime (t)\,,$ thus also 
$$
\left( \frac{\,\tau (t)\,}{t} \right)' = \,-\frac{1}{\,t^2\,} \,\big( t \, \sigma (t) \big)'. \qquad \text{Noting} \quad \lim_{t \downarrow 0} \left( \frac{\,\tau (t)\,}{t} \right)=1\,,~~~  \lim_{t \uparrow \infty} \left( \frac{\,\tau (t)\,}{t} \right)=0\,,\quad \lim_{t \downarrow 0} \left( \frac{\,\sigma (t)\,}{t} \right)=0
$$
from  \eqref{6.22},  and    integrating by parts, we obtain  the ``solution" of the above ``integral equation" as 
\begin{equation} 
\label{6.23}
\tau(M) \,= \, 2 \,M  \,\int^{\infty}_M \frac{\sigma (t)}{t^2} \,\ud t - \sigma (M)\,, \qquad M > 0 \,.
\end{equation}
Having thus arrived at the expression \eqref{6.23},  it is easy to argue  that it follows  from \eqref{6.22} in general  (i.e., without assuming continuous differentiability)  via repeated application of \textsc{Tonelli}'s theorem. The implication  $\,\lim_{M \to \infty} \sigma (M)  = 0 ~ \Longrightarrow ~ \lim_{M \to \infty} \tau (M)  = 0\,$ is now clear from \eqref{6.23}.   \qed

\subsection{Proof of Theorem  \ref{Thm6.2}}
\label{sec6.2}

   With these two results in place, and noting that it suffices to deal with $\eps =1$ in    Theorem  \ref{Thm6.2}, we start by observing  that the equivalence \,{\bf (ii)}\,$\Leftrightarrow$\,{\bf (ii)$^\prime$}\, is the subject of Lemma   \ref{Lem6.5}.  
      
   The implication \,{\bf (i)}\,$\Rightarrow$\,{\bf (ii)}\, is from  \textsc{Feller},\,Vol.\,II \cite{F}, p.\,236; and   \,{\bf (ii)}\,$\Rightarrow$\,{\bf (i)}\, is proved in \textsc{Feller}, Vol.\,II (\cite{F}, p.\,235) where it is   pointed out that, with $\, d_N   \equiv  \delta_N = \E   \big( h_1 \cdot  \mathbf{1}_{ \{ |h_1| \le   N \} }   \big)\,,$   the bound
\begin{equation} 
\label{6.21}
 \pi_N (\eta) \,= \,\P \bigg( \, \bigg| \frac{1}{\,N\,} \sum_{n=1}^N h_n - \delta_N \bigg| > \eta  \bigg)\,\le \, \tau_N + \frac{1}{\,\eta^2\,} \, \sigma_N \, \longrightarrow\,0\,, \qquad \text{as}~~ N \to \infty\,,
\end{equation}
holds for   $\eta > 0\,.$ We note  here     $\,\lim_{N \to \infty} \tau_N = 0 ~ \Rightarrow ~ \lim_{N \to \infty} \sigma_N = 0\,$ from Lemma   \ref{Lem6.5}.

We turn   to the equivalence \,{\bf (ii)$^\prime$}\,$\Leftrightarrow$\,{\bf (ii)$^{\prime \prime}$}. From Lemma   \ref{Lem6.4}, for any given $\eta \in (0,1)$ we have 
  \begin{equation} 
\label{6.25}
   \sigma_M = v_M + \frac{1}{M} \,\delta_M^2 \ge \, v_M=    \sigma_M  -  \frac{1}{M} \,\delta_M^2 \ge \sigma_M  -  \eta \big( 1 + \sigma_M \big)= \big( 1 - \eta \big) \sigma_M - \eta
\end{equation}
   for every $M \ge N_\eta$. The equivalence $\, \lim_{M \to \infty} v_M = 0 \, \Longleftrightarrow \, \lim_{M \to \infty} \sigma_M = 0\,$ follows from \eqref{6.25}, as does the equivalence \,{\bf (iii)$^\prime$}\,$\Leftrightarrow$\,{\bf (iii)$^{\prime \prime}$}.    Now,  we have shown that the  conditions  {\bf (ii)},\,{\bf (ii)$^\prime$} are equivalent, so each of them implies  also 
$\,
\lim_{M \to \infty}   \rho_M = \lim_{M \to \infty} \big( \tau_M + \sigma_M \big)= 0 \,,$ 
 namely,   {\bf (ii)$^{\prime \prime \prime}$};   and is in turn implied by it,  because all quantities in question are positive.  

 But  {\bf (ii)$^{\prime \prime \prime}$}    is   also equivalent to   
{\bf (iii)}  on account of the identity $\rho (M)   = (2/M) \int_0^M \tau (t) \, \ud t\,$ from \eqref{6.22}; and to the condition {\bf (iii)$^{\prime}$}  on account of   $\,\rho (M)   = 2 M \int_0^M (\sigma  (t) / t^2 )\, \ud t\,$ from \eqref{6.23}.
   \qed

\section{The Proof  of Theorem \ref{Thm6.1}   }   
\label{sec9}

   {\it The conditions\,\,{\bf (i)},\,{\bf  (i)$^\prime$}
are equivalent.}\,\,Indeed, for all $\eps >0\,$, $N \in \N\,,$ we have $ \,\pi_N^*( \eps) = \E \big[ \,
 \mathbold{\Pi}_N^* (\eps) \, \big]\,$ and $\, 0 \le  \mathbold{\Pi}_N^* (\eps) \le 1\,,$ $\P-$a.e. Now, convergence in probability implies $\Ll^1-$convergence under uniform boundedness (e.g., Theorem 4.5.4 in \textsc{Chung} \cite{Ch}), so   {\bf (i)} $ \Rightarrow$  {\bf  (i)$^\prime$} follows. 
 On the other hand, the reverse  implication   {\bf  (i)$^\prime$}  $ \Rightarrow$  {\bf (i)} is   
 straightforward, since $\Ll^1-\,$convergence implies convergence in probability.

For the remaining   equivalences,   
we invoke  Proposition \ref{Prop6.6} and  Corollary \ref{Cor6.7} right below.       

\begin{proposition} 
  \label{Prop6.6} 
In the context of Theorem \ref{Thm6.2} and with the notation of \eqref{6.9'}--\eqref{6.9}, there is a universal constant $\,C \in (0, \infty)$ such that, for all $\,N \in \N$ sufficiently large, we have 
 \begin{equation} 
\label{6.25a}
\pi_N \, \le \, C \cdot \big[ \, \big( \tau_N \wedge 1\big) + \big( \sigma_N \wedge 1\big)\, \big]
\end{equation}
 \begin{equation} 
\label{6.26}
  \big( \tau_N (4) \wedge 1\big) + \big( \sigma_N \wedge 1\big)\,   \le \, C \cdot \,\pi_N \,, \qquad   \big( \tau_N (4) \wedge 1\big) + \big( \,v_N \wedge 1\big)\,   \le \, C \cdot \,\pi_N \,.
\end{equation}
 \end{proposition} 

 \begin{corollary} 
  \label{Cor6.7} 
  In the context of Theorem \ref{Thm6.2} and with the notation of \eqref{6.9'}--\eqref{6.9}, the following are equivalent: 
  
  \smallskip
  \noindent
 {\rm (a)} $\,\lim_{M \to \infty} \pi_M (\eps) =0\,,$ ~for all $\eps >0\,,$
 
  \smallskip
  \noindent
 {\rm (b)} $\,\lim_{M \to \infty} \tau_M (\eps) =0\,$ ~{\rm and}~ $\,\lim_{M \to \infty} \sigma_M (\eps) =0\,,$ ~for all $\eps >0\,,$
  
  \smallskip
  \noindent
 {\rm (c)} $\,\lim_{M \to \infty} \tau_M (\eps) =0\,$ ~{\rm and}~ $\,\lim_{M \to \infty} v_M (\eps) =0\,,$ ~for all $\eps >0\,.$
   \end{corollary}

   The Corollary follows directly from    Proposition   \ref{Prop6.6}, and completes the   proof of Theorem  \ref{Thm6.1}.  \qed

\subsection{Proof of Proposition  \ref{Prop6.6}}
\label{sec6.5}

The   inequality \eqref{6.25a}   restates the  bound \eqref{6.21} (cf.\,\cite{F},\,p.\,235); in fact, the (formally stronger)   bound $\, \pi_N \, \le \, C \cdot \big[ \, \big( \tau_N \wedge 1\big) + \big( v_N \wedge 1\big)\, \big]\,$   also holds for $N \in \N\,$ sufficiently large, by the double inequality  \eqref{6.25}. This latter  shows also that the first inequality in \eqref{6.26} follows from the second.  

It remains, therefore, to prove the second inequality in \eqref{6.26}. This, in turn, follows from     Lemmata   \ref{Lem6.8}, \ref{Lem6.9}   below; these   analyze the symmetric case first,  and provide  the desired bound
$$
  \big( \tau_N (4) \wedge 1\big) + \big( \,v_N \wedge 1\big)\,   \le \,4\, \Big[ \, \big( \tau_N^{\,sym} (2) \wedge 1\big) + \big( \,v_N^{\,sym} (2) \wedge 1\big)\,\Big]\,\le\,4 C^{\,sym}\, \pi_N^{\,sym} (2)\,\le\,8 C^{\,sym}\, \pi_N 
$$
  for all $\,N \in \N$ sufficiently large. This settles the general case, and completes the proof.   \qed  

 \begin{lemma} {\bf Bounds for the Symmetric Case:} 
  \label{Lem6.8} 
  Let $\,h_1, h_2, \cdots$ be I.I.D.\,with   symmetric distribution. Recalling \eqref{6.9'}--\eqref{6.9} we have, for all $\,N \in \N\,$ sufficiently large and  some universal constant $\,C^{\,sym}>0\,,$  
   \begin{equation}
  \label{6.28}
  \tau_N \wedge 1 \le 6 \cdot \pi_N\,,  
  \end{equation}
  \begin{equation}
  \label{6.29}
  \big( \tau_N   \wedge 1\big) + \big( \sigma_N \wedge 1\big)\,   \le \, C^{\,sym} \cdot \,\pi_N \,.
\end{equation}
   \end{lemma} 
   
   \noindent
   {\it Proof:} With $\,A_N := \bigcup_{n=1}^N \big\{ \big|h_n \big| > N \big\}\,$ we have $\, \P (A_N) = 1 - \big( 1 - (\tau_N / N ) \big)^N \ge 1 - e^{- \tau_N}\,,$ thus also $\, 3\, \P(A_N) \ge   \tau_N \wedge 1\,;$ whereas,  by the symmetry assumption  and   Lemma 2 on p.\,149 of \textsc{Feller} \cite{F}, we obtain    the claim of  \eqref{6.28} as 
   $$
   \pi_N = \P \Big( \,\Big| \sum_{n=1}^N h_n \Big| > N \Big)  \ge \frac{1}{\,2\,} \, \,\P (A_N) \ge \frac{1}{\,6\,} \, \,  \big( \tau_N   \wedge 1\big) \,.
   $$

Next, we need to estimate $\sigma_N$ in terms of $\pi_N\,$; and for this, we may assume $\, \P (A_N) \le 1/2\,,$ as otherwise   \eqref{6.29} holds for the constant $\, C^{\,sym} = 12\,.$ We introduce   the probability measure 
 \begin{equation}
  \label{6.30}
 \P^{(N)} \,:=\, \frac{1}{\,1-\P(A_N)\,} \cdot \P \,\Big|_{\Omega \setminus A_N}\,,
\end{equation}
the normalized restriction of $\,\P$ to   $\,  \Omega \setminus A_N = \bigcap_{n=1}^N \big\{ \big| h_n \big| \le N \big\}\,$; and note that, under this   measure,  the functions   $\,h^{(N)}_n := h_n \cdot \mathbf{1}_{ \,\{ | h_n | \le N \}}\,,~ n=1, \cdots, N\,$ are independent,  with common distribution supported on $\big[-N,N\,\big].$  Then 
\begin{equation}
  \label{6.31}
 \pi^{(N)}_N := \P^{(N)} \Big( \,\Big| \sum_{n=1}^N h_n \Big| > N \Big)  \ge c \cdot \big( \sigma_N \wedge 1\big)
\end{equation}
holds   for some real constant $c>0$, as shown in subsection \ref{sec6.6} below. But from   \eqref{6.31}, \eqref{6.30} we deduce the inequality right below,   which  establishes   \eqref{6.29}   in conjunction with \eqref{6.28}: 
$$
\qquad \qquad 
\qquad \qquad    \pi_N = \P \Big( \,\Big| \sum_{n=1}^N h_n \Big| > N \Big)  \, \ge \, \frac{\, \pi^{(N)}_N \,}{\,2\,} \,  \ge   \,  \frac{\, c \,}{\,2\,} \,\big( \sigma_N   \wedge 1\big) \,.\qquad \qquad  \qquad \qquad  \qquad \qquad \qed
   $$
 
   \smallskip
 
 {\it  We pass now to the general case:} with $\,h_1^\pm, h_2^\pm, \cdots$   independent copies of $\,h_1, h_2, \cdots\,,$ we consider the sequence  of   ``symmetrized" versions $\,h_n^{\,sym} := h_n^{\,+} - h_n^{\,-} \,, ~ n \in \N\,,$ which are I.I.D. with   symmetric distribution. We denote then by $\,\tau^{\,sym}_t (\eps)\,,$ $\,\sigma^{\,sym}_t (\eps)\,,$  $\,\pi^{\,sym}_t (\eps)\,, \cdots $ the     so-symmetrized versions of the quantities in \eqref{6.7}, \eqref{6.8}, \eqref{6.9'}, $\cdots$ pertaining to $\,h_1^{\,sym}$.
   
    \begin{lemma} {\bf Bounds for the General Case:} 
  \label{Lem6.9} 
  Let $\,h_1, h_2, \cdots$ be I.I.D. 
  With   the notation of \eqref{6.9'}--\eqref{6.9} we have, for all $\,N \in \N$ sufficiently large, 
   \begin{equation}
  \label{6.32}
 ~~~~~~~~     v_N  \, \le    \,   2   \cdot v_{2N}^{\, sym}     \, =    \,   2   \cdot \sigma_{2N}^{\, sym}  
  \end{equation}
  \begin{equation}
  \label{6.33}
  \tau_N   \,   \le \, 4 \cdot \tau^{\,sym}_N (1/2)
\end{equation}
\begin{equation}
  \label{6.34}
 ~~~~  \pi_N   \,   \ge \, (1/2)  \cdot \pi^{\,sym}_N ( 2)\,.
\end{equation}
   \end{lemma} 
   
   \noindent
   {\it Proof:} With $\delta_N$ as in \eqref{6.6},   integers $N \ge N_0$ such that $\, \P \big( \big| h_1 \big| > N_0 \big) < 1/2\,$ (we note that this $N_0$ depends only on the distribution of $h_1$), and setting $\,q_N := 1 / \P \big( \big| h_1 \big| \le N \big) <  2\,,$ we obtain \eqref{6.32} from   
   $$
   N \,v_N = \E \Big[ \Big( h_1^+ - \delta_N \Big)^2  \mathbf{ 1}_{ \{ | h_1^+ | \le  N \} } \Big] = q_N \cdot \E \Big[ \Big( h_1^+ - \delta_N \Big)^2  \mathbf{ 1}_{ \{ | h_1^+ | \le  N, \,| h_1^-| \le  N \} } \Big]
   $$
   $$
   = \frac{\,q_N\,}{\,2\,} \cdot \E \Big[ \Big( \big( h_1^+ - \delta_N \big) - \big( h_1^- - \delta_N \big) \Big)^2  \mathbf{ 1}_{ \{ | h_1^+| \le  N, \,| h_1^-| \le  N \} } \Big]
   $$
   $$
   \le    \,\E \Big[ \Big( h_1^+ -  h_1^- \Big)^2  \mathbf{ 1}_{ \{ | h_1^+ - h_1^-| \le 2\, N \} } \Big]=  \,\E \Big[ \Big( h_1^{\, sym} \Big)^2  \mathbf{ 1}_{ \{ | h_1^{\, sym}| \le 2\, N \} } \Big]= \, 2\,N\, v_{2N}^{\, sym} \,.
   $$
 
   On the other hand, the technique of symmetrization (cf.\,Lemma 1, p.\,149 in \textsc{Feller} \cite{F}) gives also 
$$
\tau^{\,sym} = N \cdot \P \big( \big| h^+_1 - h^-_1 \big| > N \big)  >\frac{\,N\,}{2} \cdot \P \big( \big| h^+_1   \big| >  2\,N \big) =\frac{\, \tau_N (2)\,}{4}\,, \qquad \forall ~~N \ge N_0
$$
with $\, \P \big( \big| h_1 \big| > N_0 \big) < 1/2\,,$ leading to   \eqref{6.33}.    
       Finally, with $\, S^\pm_N := \sum_{n=1}^N h_n^\pm\,,$ $\, S^{sym}_N := \sum_{n=1}^N h_n^{sym}\,,$   the   inequalities 
   $\,
   \big| \big( S_N^+ / N \big) - D_N \big| \le 1 \,,~    \big| \big( S_N^- / N \big) - D_N \big| \le 1 \,,
   \,$
 lead  to $\,   \big|   S_N^{\,sym} / N    \big| \le 2\,$, and this in turn to the inequality
  $\,
  2 \cdot \P \Big( \big| \big( S_N^+ / N \big) - D_N \big| > 1 \Big)  \, \ge \, \P \big( \, \big|   S_N^{\,sym}      \big| >  2\, N \big) \,,
  $ 
  i.e., \eqref{6.34}. \qed

\subsection{Proof of the Lower Bound (8.6)}
\label{sec6.6}

The inequality   \eqref{6.31} will follow from the following result.

   \begin{lemma}  
  \label{Lem6.10}
Suppose $\, h_1, \cdots, h_N\,$ are independent functions with common, symmetric distribution supported on the interval $\,\big[-N, N\,\big]\,,$ and with $\, \sigma_N = v_N = \E \big( h_1^2 \big) / N >0\,$ in the notation of \eqref{6.8}, \eqref{6.9}. There is then a uniform constant $\,c^{sym}>0$ such that 
\begin{equation}
  \label{6.35}
 \pi_N  = \P  \Big( \,\Big| \sum_{n=1}^N h_n \Big| > N \Big)\,  \ge  \,c^{sym} \cdot    \sigma_N\,=\,c^{sym} \cdot    v_N \,.
\end{equation}
   \end{lemma}
    \noindent
   {\it Proof:} We note that $\, S_n := h_1 + \cdots + h_n\,,~ n = 1, \cdots, N\,$ is a     martingale (of its own filtration), and set $\, \tau_1 := \min \big\{ n : \big| S_n \big| \ge N \big\}\,$ with the understanding $\, \min \emptyset = \infty\,.$ With $\, \alpha : = c^{sym} \cdot    \sigma_N  \,$ the right-hand side in \eqref{6.35} and $\,\beta := \P \big( \tau_1 < \infty \big) \,,$ we have to show $\, \beta > 2\,\alpha\,;$ because then symmetry gives   \eqref{6.35}, via $ \pi_N=  \P \big( \big| S_N \big| \ge N \big) \ge  \P \big( \big| S_{\tau_1 \wedge N} \big| \ge N \big) \, \big/\, 2 \,=\, \beta \, /\,2\, >\, \alpha\,.
    $

   \smallskip
   We define now  inductively the   stopping times
   $\,
   \tau_{k+1} := \min \big\{ n  > \tau_k : \big| S_n - S_{\tau_k}\big| \ge N \big\}\,, ~ k \in \N\,
    $
and note   $\, \P \big( \tau_2 < \infty \, \big|\, \tau_1 < \infty ) \le   
\beta \,;$  we rely  on the homogeneity of the increments of $\, \big( S_n\,, \, n=1, \cdots, N \big) \,$ and the fact that, after time $\tau_1\,$, this martingale has less time to reach the barrier $N$ than when starting at  
$\, \tau_0=0\,.$ We obtain thus $\,\P \big( \tau_2 < \infty \big) \le \beta^2\,;$ as well as $\,\P \big( \tau_k < \infty \big) \le \beta^k\,, ~ k \in \N_0\,$ by  an obvious induction, and setting $\,\tau_0 \equiv 0\,.$

Now, $\, \big|   S_{\tau_1  \wedge N} \big| = \big| S_N \big| \le N \,$ holds on $\, \big\{ \tau_1 = \infty \big\}\,;$  the assumption    $\, \P \big( \big|h_1\big| \le N \big) =1\,$   implies $\, \big|   S_{\tau_k  \wedge N} \big| = \big| S_N \big| \le \big( 2\,k -1 \big)\,N \,$  on $\, \big\{ \tau_{k-1} < \infty\,, \tau_k = \infty \big\}\,;$ and $\, \big|   S_{N} - S_{ \tau_{k-1}} \big| \le   \, N\,$ holds on $\, \big\{ \tau_k = \infty \big\}\,.$ Writing 
 $
\,S_N = \sum_{k \in \N_0} \, S_N \cdot \mathbf{1}_{ \{ \tau_{k-1} < \infty\,,\, \tau_k = \infty \} }\,,
 $
 we  obtain  the inequality $\, \beta > 2\,\alpha\, $ from 
$\,
 \sigma_N\, = \,  N^{-2} \, \E(S_N^2)    \le \big( 1 - \beta \big) \sum_{k \in \N_0} \beta^{k-1}\,        \big( 2\,k -1 \big)^2 <     \beta \,/\, \big(  2\,c^{sym} \big)\,  $   for a suitable constant $\,c^{sym}>0\,,$ as long as   $\,\beta\,$ is  bounded away from 1\,. \qed

\section{The Proof of Proposition  \ref{Prop6.3}}
\label{sec6.44}

The idea  behind the proof,  is the familiar theme of ``gliding humps". In its simplest form, this is reflected in 
the following arch-example of functions $\, f_1, f_2, \cdots\,$  in $\, \Ll^0 \big( [0,1)\big)\,$ given by 
$$
f_{2^m +j} \,:= \,\mathbf{1}_{\,{\cal I}_{m,j}}\quad~ \text{where} \quad~ {\cal I}_{m,j}\,:=\, \big[ \, \big(j-1\big) \, 2^{-m},  j\, 2^{-m}\, \big)\,; \qquad m \in \N_0\,,~ ~j=1, \cdots, 2^m\,.
$$
These converge to zero in probability, but not a.e. By defining instead $\,
f_{2^m +j}^{\, \dagger} \,:= \,m \cdot \mathbf{1}_{\,{\cal I}_{m,j}}\,,$ we obtain a sequence  converging to zero in probability, yet with $\, \overline{\lim}_{n \to \infty} \big| f_n^{\, \dagger} (\omega) \big| = \infty\,$ for each $\,\omega \in [0,1)\,.$

\subsection{Proof of Proposition  \ref{Prop6.3}\,(${\mathfrak A}$)}
\label{sec6.44b}

Denoting by $\,  h^\alpha, h^\alpha_1, h^\alpha_2, \cdots\,$ an I.I.D.\,\,sequence with $\, \P ( h^\alpha = \pm N_\alpha) = \alpha / 2\,,$   $\, \P ( h^\alpha =0) = 1 - \alpha \,,$ and $\, N_\alpha = 1 / \alpha \,$ a large natural number,   
$\,
\P \Big( \Big| \sum_{n=1}^{N_\alpha} h^\alpha_n \Big| \ge N_\alpha \Big) \, > \, C\,
 $
holds for some universal constant $C \in (0,1)\,.$ 
 We may also find, for given $\eps >0\,,$ a constant $\, \rho = \rho (\eps) \in (0,1)\,$ such that,   for each  $\,\alpha >0\,,$ we have  
\begin{equation}
  \label{W1}
\P \Big( \Big| \sum_{n=1}^{N } h^\alpha_n \Big| > \eps \cdot N  \Big)\,<\, \eps\,, \qquad \text{for all} \quad N \notin \big[ \,\rho \,N_\alpha\,, (1 / \rho ) \,N_\alpha \, \big]\,.
\end{equation}
More generally, for a sequence $\,  \big( N_{\alpha_k} \big)_{k \in \N} \subset \N\,$   increasing to infinity so  fast  that   $\, \sum_{k \in \N} ( 1 / N_{\alpha_k} ) $ $< 1\,,$  and for the sequence $\,{\bm \alpha} :=  \big( \alpha_k  \big)_{k \in \N}\,$ with  $\, \alpha_k := 1 / N_{\alpha_k}\,,$ consider   I.I.D.  random variables \begin{equation}
  \label{W2}
 \text{$ h^{({\bm \alpha})}, h^{({\bm \alpha})}_1, h^{({\bm \alpha})}_2, \cdots\,$ with} \quad  \P \big( h^{({\bm \alpha})} = \pm N_{\alpha_k} \big) = \alpha_k / 2\, ~~\text{for}~~ k \in \N\,,  \quad  \P \big( h^{({\bm \alpha})} =0 \big) = 1 - \sum_{k \in \N} \alpha_k \, .
  \end{equation}
 
 We may choose   a sequence $\big( N_{\alpha_k},\, k \in \N \big) \,,$ increasing sufficiently fast to infinity, so that the intervals
\begin{equation}
  \label{W1W}
 {\cal L}_k \,:= \,\Big\{ N \in \N \,:\, \, \P \Big( \,\Big| \sum_{n=1}^N h^{\alpha_k}_n \Big| > N\,2^{-k} \Big) < 2^{-k} \Big\}
  \end{equation}
  are disjoint. Now, for each $\, x \in [0,1)\,,$ we let $\,j^{(m)}  (x)\,$ be the number in $\, \{ 1, \cdots, 2^m \}$ with  the property   $\, \big( \, j^{(m)}  (x) - 1 \big) \,2^{-m} \le x <  j^{(m)}  (x) \, 2^{-m}\,.$ Setting $\, k_m (x):= 2^m + j^{(m)}  (x)\,,$ we consider the sequence $\, \big( \alpha_{k_m (x)} \big)_{m \in \N}\,,$ and let  $\,  h^x, h^x_1, h^x_2, \cdots\,$ be the I.I.D.\,\,sequence attached to this subsequence as in \eqref{W2}. 
  
  Fixing $x \in [0,1]$ we thus have that the quantities $\, N_{\alpha_{k_m (x)}}\,\P \big( \big| h^x_1 \big| \ge N_{\alpha_{k_m (x)}} \big)\,$ as well as  the quantities $\, N^{\,-1}_{\alpha_{k_m (x)}}\,\E \Big( \big( h^x_{  1} \big)^2 {\bm 1}_{ \{ | h^x_1| \le N_{\alpha_{k_m (x)}} \} }  \Big)\,$ remain bounded away  from zero as $m \to \infty\,.$ On the other hand, when we replace in these two expressions $N_{\alpha_{k_m (x)}}$ by $N \in \N \,,$ they tend to zero as $N \to \infty\,$ except for the case where $N$ is  in one of the intervals $\, {\cal L}_{k_m (x)} \,$\ defined in and below \eqref{W1W}.  \qed

\subsection{Proof of Proposition  \ref{Prop6.3}\,(${\mathfrak B}$)}
\label{sec6.44c}

As in the proof of part (${\mathfrak A}$)\,, we construct inductively an increasing sequence $\, \big( N_k \big) = \big( N_{2^m +j} \big) \subset \N\,.$   Fixing $m \in \N$ and $\, k = 2^m +j\,$ we define also random variables $h^k$  with values in $\, \{ \pm N_k, 0 \}\,,$    with $\P \big( h^k = \pm N_k \big) = p_m / 2\,$ and $m$  sufficiently small, to wit, $\,p_m < 2^{-m}\,.$

\smallskip
The inductive step $\, m-1 \longmapsto m\,$ of this construction is described in the following Lemma, where we assume $\,N^{(m-1)}_{2^{m-1}+1}\,$ has already been defined.

 \begin{lemma} 
  \label{Lem6.8z}
For $m \in \N\,,$ there exist   integers $\,   N_{2^{m-1}+1}^{(m-1)}<  N_{1}^{(m)} <  \cdots < N_{{2^m} }^{(m)}< N_{{2^m}+1}^{(m)}\,$ and numbers $\, p_m \in (0, 2^{-m})\,,$ such that, for  $h^m_j\,$ taking  values in $\, \{ \pm N_j^{(m)}, 0 \}\,$ with $\, \P \big( h^m_j = \pm N_{j}^{(m)} \big) = p_m / 2\,,$ and setting  
\begin{equation}
  \label{W1z}
\tau^{(m,j)}_N \,:= \,N \, \cdot \, \P \, \big( \big| h_j^m \big| \ge N \big) \,, ~~~~~~~\sigma^{(m,j)}_N \,:= \,  \E \, \Big( \big(h_j^m\big)^2 \, \mathbf{1}_{ \{ | h_j^m | \le N \} } \Big) \Big/ \,N \,,
\end{equation}
we have 

 \medskip
  \noindent
 $(i)~~\tau^{(m,j)}_{N_1^{(m)}  } =m\,$ ~for $~~\, j=1, \cdots, 2^m\,,$ 
  
  \smallskip
  \noindent
  $(ii)~~\sigma^{(m,j)}_{N } < 2^{-m}\,$  ~for $~~\, j=1, \cdots, 2^m\,,$  $ N \notin \big[\,N_j^{(m)}, ~N_{j+1}^{(m)}\,\big[\,.$
      \end{lemma} 
      
       \noindent
   {\it Proof:} We choose $\,     N_{1}^{(m)} > N_{2^{m-1}+1}^{(m-1)}\,$ big enough, so that $\, p_m = m\, / \,N_{1}^{(m)} < 2^{-m}\,$; then $(i)$ follows for $j=1\,$. We note also that  $\,\sigma^{(m,1)}_{N }= \big( N_{1}^{(m)} \big)^2 \, p_m \,/\,N\,$ tends to zero as $N \to \infty\,,$ so we can find $\,     N_{2}^{(m)}\,$ with  $\,\sigma^{(m,1)}_{N_{2}^{(m)} }\,$   smaller than $2^{-m}.$ This last property holds also for all $\, N \ge N_{2}^{(m)} \,,$ while we have $\,\sigma^{(m,1)}_{N}=0\,$   for $\,N < N_{1}^{(m)}$; thus, the claim $(i)$ is established for $j=1\,.$ 
   
   We define $h^m_2$ to take values in $\, \big\{ \pm N_{2}^{(m)}, 0 \big\}$ with  $\P \big( h^m_2 = \pm N_2 \big) = p_m / 2\,,$ and observe that $(i)$ also holds for $j=2\,,$ i.e., $\,\tau^{(m,2)}_{N_1^{(m)}  } =m\,.$ Next, we choose $\,N_3^{(m)}  \,$ big enough, so that $(ii)$  holds   for $j=2\,.$  
   
   Continuing in an obvious manner, we  select $\,      N_{1}^{(m)} <  \cdots < N_{{2^m} }^{(m)} \,;$ as well as $N_{{2^m}+1}^{(m)}\,,$ whose only role is to make sure that in the   next     inductive step $\, m  \longmapsto m +1\,$ of this construction,   the number $N_{ 1}^{(m+1)}\,$ is big enough so   $\,\sigma^{(m,2^m)}_{N } < 2^{-m}\,$ holds for $\,N \ge N_{ 1}^{(m+1)}\,$.
   
   This completes the inductive step, and establishes the Lemma. \qed

    \medskip
   As in subsection \ref{sec6.44b}, we arrange  now   these   $\, \Big( \big( N_j^{(m)} \big)_{j=1}^{2^m} \Big)_{m  \in \N_0}\,$ as 
  $\, \big( N_k \big)_{k=2}^{\infty} =\Big( \big( N_{j + 2^m} \big)_{j=1}^{2^m} \Big)_{m  \in \N_0} \,$; associate  to each  $\,x \in [0,1)\, $ the integer $\,J^{(m)} (x)\,$ as in subsection \ref{sec6.44b} above; and let   $\, \big( h^x_n \big)_{n \in \N}\,$ be I.I.D. taking the values $\, 0, \, \pm N_{J^{(1)} (x)}^{(1)}, \, \pm N_{J^{(2)} (x)}^{(2)} , \cdots\,$ on disjoint sets,  
  with $\, \P \big( h^x = \pm N_{J^{(m)} (x)}^{(m)} \big) = p_m / 2\,,~ m \in \N\,.$ This latter is possible, because $\, p_m <2^{-m}\,.$ Once again, we   consider       exchangeable $\,  g^X_1,g^X_2,\cdots\,$ with the   property that,    for $X$ uniformly distributed  on [0,1), the conditional distribution of $\,   g^X_1,g^X_2,\cdots\,,$ given $\, \{ \,X = x\, \}\,,$ is that of the sequence $\,    h^x_1, h^x_2, \cdots\,$.
 
 \smallskip
 For each $\,m \in \N\,$ we consider the first element $\,\widehat{N}_M := N_{2^m+1}= N^{(m)}_1\,$ of the above inductive construction.   Lemma \ref{Lem6.8z}\,$(i)$ shows that $\,\widehat{N}_M \cdot \P \big( \big| g^X_{m} \big| \ge \widehat{N}_M \, \big|\, X=x \big) \ge \widehat{N}_M \cdot \P \big( \big| h^x_{m} \big| \ge \widehat{N}_M   \big) = m\,$ holds  for each $x \in [0,1).$  On the other hand, for every integer $N$ between $ N_1^{(m)}$ and $ N_1^{(m+1)},$ there is at most one $\widehat{j} \in \big\{ 1, \cdots, 2^m \big\}$ so  that, for $x \in \big[ \big( \,\widehat{j} - 1 \big)  \,2^{-m}, \,\widehat{j}  \,2^{-m} \big)$, 
  $\sigma^{(m,\widehat{j})}_{N }$ of   \eqref{W1z} is big; while, for   $j \neq \widehat{j}$, we have $\sigma^{(m,j)}_{N } < 2^{-m}\,.$ This   leads to $(ii)$, and completes the proof of part (${\mathfrak B}$). \qed

\subsection{Proof of Proposition  \ref{Prop6.3}\,(${\mathfrak C}$)}

We carry out   a construction similar to that in the proof of part (${\mathfrak B}$)\,, defining   inductively an increasing sequence $\, \big( N_k \big) = \big( N_{2^m +j} \big) \,$ of natural numbers. Fixing again $m \in \N$ we shall construct this time $\, \big( N_k \big)_{k=2}^{\infty}\, := \Big( \big( N_{2^m +j} \big)_{j=1}^{2^m} \Big)_{m  \in \N } \equiv 
\Big( \big( N_j^{(m)} \big)_{j=1}^{2^m} \Big)_{m  \in \N }\,$ via {\it backwards  induction:} 

We   choose first a (very large) number $N_{2^m  }^{(m)} \,,$ then inductively a decreasing string of natural numbers $\,N_{{2^m} }^{(m)} > N_{{2^m} -1}^{(m)} > \cdots > N_{ 1}^{(m)} \,,$ all the while making sure that $N_{ 1}^{(m)}$ is  bigger than a threshold $N_{2^{m-1}}^{(m-1)}$ given from the previous inductive step $(m-1)$. 

 \begin{lemma} 
  \label{Lem6.6z}
Consider natural numbers $\,m$ and $\,N_{2^{m-1}}^{(m-1)}\,,$ as well as a number $\,\delta >0$ and a rational $\eps >0\,.$    There exist then natural numbers $\,N_{ 2^{\,m -1}}^{(m-1)} < N_{1}^{(m)} < \cdots < N_{{2^m} }^{(m)} \,,$ and positive numbers $\, p^{(m)}_1, \cdots, p^{(m)}_{2^m}\,$ with $\, \sum_{j=1}^{2^m} p^{(m)}_j < \delta\,,$ such that 
\begin{equation}
  \label{W2z}
 p^{(m)}_j \, \Big( N^{(m)}_j \Big)^2 \,=\, m \cdot  N^{(m)}_{2^m}\,,\qquad j=1, \cdots, 2^m\,,
\end{equation}
\begin{equation}
  \label{W3z}
p^{(m)}_j \,  N^{(m)}_{j-1} \,=\, \eps\,,\qquad j=	2, \cdots, 2^m\,.
\end{equation}
   \end{lemma} 
   \noindent
 {\it Proof:} The above inductive relationships   \eqref{W2z},\,\eqref{W3z} hold if, and only if, the pairs $\, \big( N^{(m)}_j, p^{(m)}_j \big)\,,~ j=1, \cdots, 2^m\,$ satisfy the following explicit relations:
$$
N^{(m)}_j = N^{(m)}_{2^m} \Big( \frac{\eps}{m} \Big)^{2^{\,2^m -j} -1 }\,, ~~~~ p^{(m)}_j = \frac{m}{N^{(m)}_{2^m}} \Big( \frac{m}{\eps} \Big)^{2^{\,2^m  }   }  \,,\qquad j=1, \cdots, 2^m\,.
$$ 
 As $\eps >0$ is assumed rational, we can choose  a sufficiently large multiple $\, N^{(m)}_{2^m}\,$ of $\big(  m\,/\eps  \big)^{2^{\,2^m  }   } \,,$ so that $\,N^{(m)}_1\,$ is a natural number bigger than the given  $\,N_{ 2^{\,m -1}}^{(m-1)}\,$ and such that 
  $\, \sum_{j=1}^{2^m} p^{(m)}_j < \delta\,$ holds. \qed
 
 \medskip
 As in the proof of part (${\mathfrak B}$), the Lemma that follows describes the inductive step $\, m-1 \longmapsto m\,,$  where we assume that $\,N^{(m-1)}_{2^{m-1} }\,$ has already been defined.

 \begin{lemma} 
  \label{Lem6.9z}
For $\,m \in \N\,$ and $\, (\eps, \delta) \in (0, \infty)^2\,,$ there exist   integers $\,   N_{2^{m-1} }^{(m-1)}<  N_{1}^{(m)} <  \cdots < N_{{2^m} }^{(m)} \,,$ as well as positive numbers $\, p^{(m)}_1, \cdots, p^{(m)}_{2^m}\,$ with $\, \sum_{j=1}^{2^m} p^{(m)}_j < \delta\,,$ and random variables $\,h^m_j\,$ with  values in $\, \big\{ \pm N_j^{(m)}, 0 \big\}\,$ and   $\, \P \big( h^m_j = \pm N_{j}^{(m)} \big) = p^{(m)}_j / 2\,,$ 
so that  
\begin{equation}
  \label{W1za}
\tau^{(m,j)}_N \,:= \,N \, \cdot \, \P \, \big( \big| h_j^m \big| > N \big) \,, ~~~~~~~\sigma^{(m,j)}_N \,:= \,  \E \, \Big( \big(h_j^m\big)^2 \, \mathbf{1}_{ \{ | h_j^m | \le N \} } \Big) \Big/ \,N 
\end{equation}
satisfy  

 \medskip
  \noindent
 $(i)~~\tau^{(m,j)}_{N   } =0\,$ ~~for $~\, j=1, \cdots, 2^m\,,~~ N \ge N_j^{(m)}\,,$ 
 
  \smallskip
  \noindent
 $(i)^{\prime}~~\tau^{(m,j)}_{N  } \le \eps^\ell \cdot \tau^{(m,j+\ell)}_{N  } \,$ ~~for $~\, j=1, \cdots, 2^m\,,~\, \ell=1, \cdots, 2^m -j\,,~~ N \le N_j^{(m)}\,,$ 
  
  \smallskip
  \noindent
  $(ii)~~\sigma^{(m,j)}_{N } = \big( N_{2^m}^{(m)} \big)^{-1} p^{(m)}_j \, \big( N_{j}^{(m)} \big)^{2} = m\,$  ~~for $~\, j=1, \cdots, 2^m\,,~ N = N_{2^m}^{(m)} \,.$
      \end{lemma} 
      
       \noindent
   {\it Proof:} Part $(i)$ is obvious;   $(ii)$ follows from \eqref{W3z};  and $(iii)$   from applying \eqref{W2z} inductively. \qed

 \medskip
 
   Again as in subsection \ref{sec6.44b}, we arrange      the    $\, \Big( \big( N_j^{(m)} \big)_{j=1}^{2^m} \Big)_{m  \in \N_0}\,$ as 
  $\, \big( N_k \big)_{k=2}^{\infty} =\Big( \big( N_{j + 2^m} \big)_{j=1}^{2^m} \Big)_{m  \in \N_0} \,$; associate  to each  $\,x \in [0,1)\, $ an integer $\,J^{(m)} (x)\,$; and let   $\, \big( h^x_n \big)_{n \in \N}\,$ be I.I.D. taking the values $\, 0, \, \pm N_{J^{(1)} (x)}^{(1)}, \, \pm N_{J^{(2)} (x)}^{(2)} , \cdots\,$ on disjoint sets,    with $\, \P \big( h^x = \pm N_{J^{(m)} (x)}^{(m)} \big) = p_m / 2\,,~ m \in \N\,.$ We   summon   also    exchangeable $\,  g^X_1,g^X_2,\cdots\,$ with the   property that,    for $X$ uniformly distributed  on [0,1), the conditional distribution of $\,   g^X_1,g^X_2,\cdots\,,$ given $\, \{ \,X = x\, \}\,,$ is that of the sequence $\,    h^x_1, h^x_2, \cdots\,$.

\smallskip
For each $m \in \N$ we consider now the last element $\, M(m)  :=   N_{{2^m} }^{(m)}\,$ of the above string $\,N_{1}^{(m)} < \cdots < N_{{2^m} }^{(m)} \,.$ It follows from Lemma   \ref{Lem6.9z}\,$(ii)$ that property $(c)$ of Proposition \ref{Prop6.3}\,(${\mathfrak C}$) holds as $\, \P \big( \Sigma_{M (m)} \ge m \big) =1\,;$ this implies also property $(a)$  there, i.e., the failure of the WLLN. 

\smallskip
 We still have to show property $(b)$, i.e., $\,\lim_{N \to \infty}  T_N =0\,,$ \,in $\P-$probability. With $m \in \N$ and $N \in \big( N^{(m-1)}_{2^{m-1}}, N^{(m)}_{2^{m}} \big]$ there is at most one interval of the form $\, \big( \,N^{(m)}_j, N^{(m)}_{j+1} \, \big]\,$ with $\, j=1, \cdots, 2^m -1\,$ such that $\, N \in \big( \,N^{(m)}_j, N^{(m)}_{j+1} \, \big]\,.$ For such $N$ and $\ell \neq j\,,$ we have then $\, \tau_N \big( g^X_1 \big) < (1/m) = \eps\,$ on the event $\, \big\{  (j-1) \,2^{-m} < X \le  j \,2^{-m} \big\} \,$. This proves $\,\lim_{N \to \infty}  T_N =0\,,$ \,in $\P-$probability.    \qed

\section{The Proof  of Theorem  \ref{Thm3.2}}
\label{sec10}

The equivalence of the conditions\,\,{\bf (i)},\,{\bf  (i)$^\prime$} is shown exactly as   in the first paragraph of section \ref{sec9}, for the proof of Theorem  \ref{Thm6.1}. We need to prove the equivalence of conditions {\bf (i)} and {\bf (ii)}. 

\subsection{The Implication {\bf (i)}$\,\Rightarrow \,${\bf (ii)}}
\label{sec10a}

Suppose that there exists a sequence of ``correctors" $D_1, D_2, \cdots\,,$ measurable with respect to the tail $\sigma-$algebra ${\cal T}$ of \eqref{A6a}, such that the Weak Law of Large Numbers 
 \eqref{D2a} holds  for some subsequence $\,f_{k_1}, f_{k_2}, \cdots\,  $  and  all its subsequences.   Then, from Lemma   \ref{Lem4.1}, this  
 $\,f_{k_1}, f_{k_2}, \cdots\,  $ is bounded in $\Ll^0\,$; so we may assume  it is also determining, with associated exchangeable approximating sequence   $\, g_1, g_2, \cdots\,$ as in \eqref{A7}. On the strength of Proposition   \ref{Prop3.1}(i), the exchangeable   $\, g_1, g_2, \cdots\,$ satisfy  this Weak Law of Large Numbers  as well; in other words,  \eqref{6.10} holds. 
 
 Hence, we may apply Theorem \ref{Thm6.1} to obtain all the properties \eqref{6.12}--\eqref{6.16a} listed there,  for the ``statistics" \eqref{6.3}--\eqref{6.5} of the (random) conditional distribution ${\bm \mu}$ of $g_1,$ given the tail $\sigma-$algebra ${\cal T}_*$  of \eqref{6.1}. In particular, the conditions of \eqref{6.12} with $\eps =1$ lead now to those in \eqref{B00}.

\subsection{The Implication {\bf (ii)}$\,\Rightarrow \,${\bf (i)}}

We assume that the determining subsequence $\,f_{k_1}, f_{k_2}, \cdots,$ with limit random probability distribution  ${\bm \mu} $     as in   and below \eqref{B0}, satisfies the conditions in \eqref{B00}; in particular,   boundedness in $\Ll^0$.      We select  a sequence of natural numbers $\, J_N, ~N \in \N\,$   increasing to infinity so slowly, that 
  \begin{equation}
  \label{J1}
\lim_{N \to \infty} \bigg(  \frac{1}{\,N\,} \sum_{n=J_N}^N f_{k_n}   - D_N \bigg) =\,0 \quad \text{in} ~\P\text{\,--\,probability}
  \end{equation}
implies the desired \eqref{D2a}; and set out to prove  \eqref{J1}.

We summon now Theorem 2 of \cite{BerPet} to obtain,  for each $N \in \N\,$ and after passing to an appropriate subsequence, a partition $\, \big( A^{(N)}_i \big)_{i=0}^{r_N}\,$ of the space  $\Omega$ with $\,\P \big( A^{(N)}_0 \big) \le 2^{-N}\,$ and  such that, for each   remaining set  $\, A = A^{(N)}_i\,,~i=1, \cdots, r_N\,,$ there exists a sequence $\, h^{(A)}_j\,,~j \in \N\,$ of I.I.D.\,\,random variables with common distribution   $\mu^A$ and  the property 
\begin{equation}
  \label{J2}
\P^A \Big(\, \big| f_j - h^{(A)}_j \big| > 2^{-N} \,\Big) \le 2^{-N} \,, \qquad j=J_N, J_{N+1}, \cdots\,.
  \end{equation}
Here, we have renamed as $\big(f_j, ~j \in \N\big)$ the determining subsequence $f_{k_1}, f_{k_2}, \cdots$; and  denoted by $\P^A\,$ the conditional probability measure $\, \P^A (\cdot) = \P ( \cdot \,\cap A)\, / \,\P (A)\,,$   by $\mu^A$ the  limiting distribution    of the $f_1, f_2, \cdots$ on the set $A\,.$    The inequality (\ref{6.21}) (also on page 235 of \textsc{Feller},\,Vol.\,II \cite{F}) gives now 
\begin{equation}
  \label{J3}
\P^A \bigg(\, \Big| \, \frac{1}{\,N\,} \sum_{j={J_N}}^N h^{(A)}_j - \delta^{(A)}_N \, \Big|   > \eps \bigg) \le N \cdot \mu^A \big( \R \setminus [-N,N]\big)+ \frac{1}{N \eps^2} \int_{[-N,N]} x^2\, \mu^A (\ud x)
  \end{equation}
for the correctors
\begin{equation}
  \label{J4}
\delta^{(A)}_N\,:=\, \E^{\P^A} \Big( h^{(A)}_j \cdot \mathbf{1}_{ \{ |  h^{(A)}_j| \le N \}} \Big)\,, \quad N \in \N\,.
  \end{equation}
Whereas,  on account of \eqref{J3},\,\eqref{J2}, we deduce  for each atom $\, A = A^{(N)}_i\,,~i=1, \cdots, r_N\,$ the bound 
\begin{equation}
  \label{J5}
\P^A \bigg(\, \Big| \, \frac{1}{\,N\,} \sum_{j={J_N}}^N f_j - \delta^{(A)}_N \, \Big|   > \eps + 2^{-N} \bigg) \le N \cdot \mu^A \big( \R \setminus [-N,N]\big)+ \frac{1}{N \eps^2} \int_{[-N,N]} x^2\, \ud \mu^A (x)+  2^{-N}.
  \end{equation}
  
  We denote now by $\G_N$ the $\sigma-$algebra generated by the sets   $\,   A^{(N)}_i\,,~i=  1, \cdots, r_N\,,$ and observe that the construction in \cite{BerPet} can be made to guarantee the filtration structure $\,\G_1 \subseteq \G_2 \subseteq \cdots  \,;$ in  this manner,   \eqref{J5} becomes
\begin{equation}
  \label{J6}
\P \bigg(\, \Big| \, \frac{1}{\,N\,} \sum_{j={J_N}}^N f_j - D_N \, \Big|   > \eps + 2^{-N} \, \bigg|\, \G_N \bigg)~~~~~~~~~~~~~~~~~~~~~~~~~~~~~~~~~~~~~~~~~~~~~~~~~~~~~~~~
  \end{equation}
  $$
~~~~~~~~~~~~~~~  \le \, \E \Big( N \, {\bm \mu} \big( \R \setminus [-N,N], \cdot  \, \big) \, \Big|\, \G_N \Big)\,+\, \E \bigg( \frac{1}{N \eps^2} \int_{[-N,N]} x^2\, {\bm \mu} (\ud x, \cdot ) \, \bigg|\, \G_N \bigg) +    2^{-N}
  $$
        with   ``randomized correctors" 
  \begin{equation}
  \label{J7}
   D_N \, := \,\sum_{i=1}^{r_N} \,\delta^{(A^{(N)}_i)}_N \cdot \mathbf{1}_{A^{(N)}_i}\,
     \end{equation}
in the notation of     \eqref{J4}.  Whereas, passing in \eqref{J6} to a suitable subsequence of $\,f_1, f_2, \cdots$  and relabelling, this inequality remains valid when conditioning with respect to $\G_{N}$ in \eqref{J6} is replaced by  conditioning with respect to   $\,\G_{s_N}$, 
where $ \, \big( s_N \big)_{N \in \N}\,$ grows as rapidly as desired. 
    
Finally, we argue that 
  $$
  \P \bigg(\, \Big| \, \frac{1}{\,N\,} \sum_{j={J_N}}^N f_j - D_N \, \Big|   > \eps + 2^{-N} \, \bigg|\, \G_{s_N} \bigg)
   \le \,   N \, {\bm \mu} \big( \R \setminus [-N,N], \cdot  \, \big)  \,+\,   \frac{1}{N \eps^2} \int_{[-N,N]} x^2\, {\bm \mu} (\ud x, \cdot )   +   R_N\,
  $$
holds, where $\, \lim_{N \to \infty} R_N=0\,$ in probability. Indeed, recalling the assumption $(ii)$ and   \eqref{J4},   \eqref{J7},   invoking the martingale convergence theorem, and always assuming that  $ \, \big( s_N \big)_{N \in \N}\,$ grows  sufficiently  rapidly, we deduce    that the left-hand side of the above display, as well as the difference $\, D_N - \int_{[-N,N]} x \, {\bm \mu} (\ud x, \cdot )\,$, converge in probability to zero as $N \to \infty\,.$ 
 
 \smallskip
 Integration leads now to the desired conclusion \eqref{J1}. \qed

 \medskip
  

   \end{document}